\pdfoutput = 1
\documentclass[runningheads]{llncs}
\usepackage{graphicx}
\usepackage[utf8]{inputenc} % allow utf-8 input
\usepackage[T1]{fontenc}    % use 8-bit T1 fonts
\usepackage{hyperref}       % hyperlinks
\usepackage{url}            % simple URL typesetting
\usepackage{booktabs}       % professional-quality tables
\usepackage{amsfonts}       % blackboard math symbols
\usepackage{nicefrac}       % compact symbols for 1/2, etc.
\usepackage{microtype}      % microtypography
\usepackage{xcolor}         % colors

\usepackage{graphicx}
\usepackage{subfigure}

\usepackage{amssymb}
\usepackage{amsmath}
\usepackage{mathtools}
\usepackage{multicol}
\usepackage{multirow}

\usepackage{color}

\usepackage{array}
\usepackage{algorithm}
\usepackage{algorithmic}

\usepackage[shortlabels]{enumitem}

\usepackage[symbol]{footmisc} % footnote

\def\M{\mathcal{M}}

\def\L{\mathcal{L}}
\def\D{\mathrm{D}}

\def\bzero{{\mathbf 0}}

\def\bSigma{{\mathbf \Sigma}}
\def\ker{{\mathrm{Ker}}}
\def\bOmega{{\mathbf \Omega}}

\def\ba{{\mathbf a}}

\def\be{{\mathbf e}}

\def\bw{{\mathbf w}}
\def\bx{{\mathbf x}}
\def\by{{\mathbf y}}

\def\bA{\mathbf A}
\def\bB{\mathbf B}
\def\bC{\mathbf C}
\def\bD{\mathbf D}
\def\bE{\mathbf E}

\def\bH{\mathbf H}
\def\bI{{\mathbf I}}
\def\bK{{\mathbf K}}
\def\bL{\mathbf L}
\def\bM{\mathbf M}
\def\bO{\mathbf O}
\def\bP{{\mathbf P}}

\def\bR{{\mathbf R}}
\def\bS{\mathbf S}
\def\bT{\mathbf T}
\def\bU{{\mathbf U}}
\def\bV{{\mathbf V}}
\def\bW{{\mathbf W}}
\def\bX{{\mathbf X}}
\def\bY{{\mathbf Y}}
\def\bZ{{\mathbf Z}}

\def\sR{{\mathbb R}}
\def\sS{{\mathbb S}}

\def\gC{{\mathcal{C}}}

\def\gN{{\mathcal{N}}}
\def\gP{{\mathcal{P}}}

\def\gH{{\mathcal{H}}}
\def\gV{{\mathcal{V}}}

\def\vec{{\mathrm{vec}}}

\def\grad{{\mathrm{grad}}}
\def\hess{{\mathrm{Hess}}}

\newcommand{\trace}{\mathrm{tr}}

\DeclareMathOperator*{\argmax}{arg\,max}
\DeclareMathOperator*{\argmin}{arg\,min}

\title{{Learning with Symmetric Positive Definite Matrices via Generalized Bures-Wasserstein Geometry
}}

\titlerunning{Learning with Generalized Bures-Wasserstein Geometry}
\author{Andi Han\inst{1} \and
Bamdev Mishra\inst{2} \and
Pratik Jawanpuria\inst{2} \and
Junbin Gao\inst{1}}
\institute{}

\authorrunning{}

\institute{University of Sydney \\ \email{\{andi.han,junbin.gao\}@sydney.edu.au } \and
Microsoft India\\
\email{\{bamdevm, pratik.jawanpuria\}@microsoft.com}}

\begin{document}
%
% \title{Contribution Title\thanks{Supported by organization x.}}
\maketitle              % typeset the header of the contribution
\begin{abstract}
Learning with symmetric positive definite (SPD) matrices has many applications in machine learning. Consequently, understanding the Riemannian geometry of SPD matrices has attracted much attention lately. A particular Riemannian geometry of interest is the recently proposed Bures-Wasserstein (BW) geometry which builds on the Wasserstein distance between the Gaussian densities. In this paper, we propose a novel generalization of the BW geometry, which we call the GBW geometry. The proposed generalization is parameterized by a symmetric positive definite matrix $\bM$ such that when $\bM = \bI$, we recover the BW geometry. We provide a rigorous treatment to study various differential geometric notions on the proposed novel generalized geometry which makes it amenable to various machine learning applications. We also present experiments that illustrate the efficacy of the proposed GBW geometry over the BW geometry.

\end{abstract}
%
%
%

% % \begin{keywords}
% % Riemannian geometry, SPD matrices, Bures-Wasserstein
% % \end{keywords}

% \input{manuscript_gsi}
\section{Introduction}
Symmetric positive definite (SPD) matrices play a fundamental role in various fields of machine learning, such as metric learning \cite{huang2015log}, signal processing \cite{brooks2019exploring}, sparse coding \cite{cherian2016riemannian,harandi2015sparse}, computer vision \cite{guillaumin2009you,mahadevan19a}, and medical imaging \cite{pennec2006riemannian,lotte2018review}, etc. The set of SPD matrices, denoted as $\sS_{++}^n$, is a subset of the Euclidean space $\sR^{n (n+1)/2}$. To measure the (dis)similarity between SPD matrices, one needs to assign a metric (an inner product structure on the tangent space) on $\sS_{++}^n$, which yields a Riemannian manifold. 
Consequently, various  Riemannian metrics have been studied such as the affine-invariant \cite{bhatia2009positive,pennec2006riemannian}, Log-Euclidean \cite{arsigny2006log}, and Log-Cholesky \cite{lin2019riemannian} metrics, and those induced from symmetric divergences  \cite{sra2012new,sra2016positive}. 
%Consequently, there exist a number of Riemannian metrics such as the affine-invariant \cite{bhatia2009positive,pennec2006riemannian,thanwerdas2019affine}, Log-Euclidean \cite{arsigny2006log,arsigny2007geometric}, and Log-Cholesky \cite{lin2019riemannian} metrics, to name a few. 
%There also exist metrics induced from symmetric divergences  \cite{sra2012new,sra2016positive,sra2021metrics}. 
Different metrics lead to different differential structures on the SPD matrices, and therefore, picking the ``right''  one depends on the application at hand. Indeed, the choice of metric has profound effect on the performance of learning algorithms \cite{mishra2016riemannian,shustin2019preconditioned,han2021riemannian}.

The Bures-Wasserstein (BW) metric and its geometry for SPD matrices have lately gained popularity, especially in machine learning applications~\cite{bhatia2019bures,malago2018wasserstein,van2020bures} such as statistical optimal transport \cite{bhatia2019bures}, computer graphics \cite{rabin2011wasserstein}, neural sciences \cite{gramfort2015fast}, and evolutionary biology \cite{demetci2020gromov}, among others. It also connects to the theory of optimal transport and the $L_2$-Wasserstein distance between zero-centered Gaussian densities \cite{bhatia2019bures}. More recently, \cite{han2021riemannian} analyzes the BW and the affine-invariant (AI) geometries in SPD learning problems and 
compare their advantages/disadvantages in various machine learning applications. 
%discusses cases where one of them performs better than the other. 

% the BW geometry helps better in robust learning of SPD matrices in various applications compared to the AI geometry. In particular, \cite{han2021riemannian}

% \alertPJ{Can we allude to some issues of BW geometry, discussed in \cite{han2021riemannian}, which can be mitigated by GBW? This can help as a soft motivation.}

In this paper, we propose a natural generalization of the BW metric by scaling SPD matrices with a given parameter SPD matrix $\bM$. The introduction of $\bM$ gives flexibility to the BW metric.
Choosing $\bM$ is equivalent to choosing a suitable metric for learning tasks on SPD matrices. For example, a proper choice of $\bM$ can lead to faster convergence of algorithms for certain class of optimization problems (see more discussions in Section \ref{sec:experiments}).
% Choosing $\bM$ may also be viewed as choosing a suitable metric that leads to faster convergence of algorithms for certain class of optimization problems. 
Indeed, when $\bM = \bI$, the generalized metric reduces to the BW metric for SPD matrices. When $\bM = \bX$, the proposed metric coincides locally with the AI metric, i.e., around the neighbourhood of a SPD matrix $\bX$. The proposed generalized metric allows to connect the BW and AI metrics (locally) with different choices of $\bM$. The following are our contributions.
\begin{itemize}
    \item We propose a novel generalized BW (GBW) metric by generalizing the Lyapunov operation in the BW metric (Section~\ref{sec:motivations}). In addition, it can also be viewed as a generalized Procrustes distance and also as the Wasserstein distance with Mahalanobis cost metric for Gaussians. %This is discussed in Section~\ref{sec:motivations}. 
    
    \item The GBW metric leads to a Riemannian geometry for SPD matrices. In Section~\ref{sec:bw_geometry}, we derive various Riemannian operations like geodesics, exponential and logarithm maps, Levi-Civita connection. We show that they are also natural generalizations of operations with the BW geometry. Section~\ref{sec:applications} derives Riemannian optimization ingredients under the proposed geometry.
    
    \item In Section \ref{sec:experiments}, we show the usefulness of the GBW geometry in the applications of covariance estimation and Gaussian mixture models.
    
\end{itemize}

\section{Generalized Bures-Wasserstein metric} \label{sec:motivations}

The Bures-Wasserstein (BW) distance is defined as
\begin{equation}
    d_{\rm bw}(\bX, \bY) = \sqrt{\trace(\bX) + \trace(\bY) - 2 \trace(\bX \bY)^{1/2} }, \label{wasserstein_dist}
\end{equation}
where $\bX$ and $\bY$ are SPD matrices and $\trace(\bX)^{1/2}$ denotes the trace of the matrix square root. It has been shown in \cite{bhatia2019bures,malago2018wasserstein} that the BW distance \eqref{wasserstein_dist} induces a Riemannian metric and geometry on the manifold of SPD matrices. The BW metric that leads to the distance (\ref{wasserstein_dist}) is defined as
\begin{equation}\label{bw_metric}
    g_{\rm bw}(\bU, \bV) =   \frac{1}{2}\trace( \L_\bX[\bU] \bV ) = \frac{1}{2}\vec(\bU)^\top (\bX \otimes \bI + \bI \otimes \bX)^{-1} \vec(\bV), 
\end{equation}
where $\bU, \bV$ on $T_\bX\sS^n_{++}$ are the symmetric matrices and the Lyapunov operator $\L_\bX[\bU]$ is defined as the solution of the matrix equation $\bX \L_\bX[\bU] \allowbreak + \L_{\bX}[\bU]\bX = \bU$ for $\bU \in \sS^n$ (which is the set of symmetric matrices of size $n\times n$). Here, $\vec(\bU)$ and $\vec(\bV)$ are the vectorization of matrices $\bU$ and $\bV$, respectively, and $\otimes$ denotes the Kronecker product. 
Our proposed GBW metric generalizes (\ref{bw_metric}) and is parameterized by a given $\bM \in \sS_{++}^n$ as
\begin{equation}
\begin{array}{lll}\label{gbw_metric}
    g_{\rm gbw}(\bU, \bV) 
    % = \langle \bU, \bV \rangle_{\rm gbw} 
    = \frac{1}{2}  \trace(\L_{\bX,\bM}[\bU]\bV)
    = \frac{1}{2} \vec(\bU)^\top (\bX \otimes \bM + \bM \otimes \bX)^{-1} \vec(\bV), 
    \end{array}
\end{equation}
where $\L_{\bX,\bM}[\bU]$ is the generalized Lyapunov operator, defined as the solution to the linear matrix equation $\bX\L_{\bX,\bM}[\bU]\bM + \bM \L_{\bX,\bM}[\bU] \bX = \bU$. Similar to the special Lyapunov operator, the solution is symmetric given that $\bX, \bM \in \sS^{n}_{++}$ and $\bU \in \sS^n$. As we show later that the Riemannian distance associated with the GBW metric is derived as 
\begin{equation}\label{gbw_distance}
\begin{array}{lll}
    d_{\rm gbw}(\bX, \bY) = 
    \sqrt{ \trace(\bM^{-1}\bX) + \trace(\bM^{-1} \bY) - 2 \trace(\bX\bM^{-1}\bY\bM^{-1})^{1/2} },
    \end{array}
\end{equation}
which can be seen as the BW distance \eqref{wasserstein_dist} between $\bM^{-1/2} \bX \bM^{-1/2}$ and $\allowbreak \bM^{-1/2} \allowbreak \bY \bM^{-1/2}$. 
% It should be noted that 
% \[
% \trace(\bX^{1/2}\bM^{-1}\bY\bM^{-1}\bX^{1/2})^{1/2}  = \trace(\bX \bM^{-1}\bY\bM^{-1})^{1/2} = \trace(\bM^{-1}\bY\bM^{-1}\bX)^{1/2}.
% \]
Note that the affine-invariant metric \cite{bhatia2009positive} is given by $g_{\rm ai}(\bU, \bV) = \vec(\bU)^\top (\bX \otimes \bX)^{-1} \vec(\bV)$.  Clearly, the proposed metric (\ref{gbw_metric}) coincides locally with the affine-invariant (AI) metric when $\bM = \bX$, i.e., around the neighbourhood of $\bX$. (Implications of this observation are discussed later in experiments.) 

Below, we show that the same GBW distance \eqref{gbw_distance} is realized under various contexts naturally. In those cases, the Euclidean norm, denoted by $\| \cdot\|_2$ is replaced with the more general Mahalanobis norm defined as $\| \bX \|_{\bM^{-1}} := \sqrt{\trace(\bX^\top \bM^{-1} \bX)}$. 

% Hence, the GBW metric bridges the gap between BW and (locally) AI metrics with different choices of $\bM$. \alertPJ{Can this point also be mentioned in introduction - as additional motivation for such investigation.} Implications of this observation are discussed for log-determinant optimization in Section \ref{sec:log_det_optimization}. 

% \subsection{Orthogonal Procrustes problem}
\paragraph{\textbf{Orthogonal Procrustes problem:}}
Any SPD matrix $\bX \in \sS_{++}^{n}$ can be factorized as $\bX = \bP \bP^\top$ for $\bP \in {\rm M}(n)$, the set of invertible matrices. Such a factorization is invariant under the action of the orthogonal group $O(n)$, the set of orthogonal matrices. That is, for any $\bO \in O(n)$, $\bP\bO$ is also a valid parameterization. In \cite{bhatia2019bures}, the BW distance is verified as the extreme solution of the orthogonal Procrustes problem where $\bP$ is set to be $\bX^{1/2}$, i.e., $d_{\rm bw}(\bX, \bY) = \min_{\bO \in O(n)}\| \bX^{1/2} - \bY^{1/2} \bO \|_2$. We can show that the GBW distance is obtained as the solution to the same orthogonal Procrustes problem in the Mahalanobis norm parameterized by $\bM^{-1}$.

\begin{proposition}
\label{gbw_procrustes}
$d_{\rm gbw}(\bX, \bY) = \min_{\bO \in O(n)} \| \bX^{1/2} \allowbreak- \bY^{1/2} \bO \|_{\bM^{-1}}.$
\end{proposition}

% \subsection{Wasserstein distance and optimal transport}

\paragraph{\textbf{Wasserstein distance and optimal transport:}}\label{subsec:wasserstein-distance-optimal-transport}

To demonstrate the connection of the GBW distance to the Wasserstein distance, recall that the $L_2$-Wasserstein distance between two probability measures $\mu, \nu$ with finite second moments is $W^2(\mu, \nu) = \inf_{\bx \sim \mu, \by \sim \nu} \mathbb{E} \| \bx - \by \|^2_2 = \inf_{\gamma \sim \Gamma(\mu, \nu)} \int_{\sR^n \times \sR^n} \|\bx - \by \|^2_2 d\gamma(\bx, \by)$,
where $\Gamma(\mu, \nu)$ is the set of all probability measures with marginals $\mu, \nu$. It is well known that the $L_2$-Wasserstein distance between two zero-centered Gaussian distributions is equal to the BW distance between their covariance matrices~\cite{bhatia2019bures,computationalOTbook,van2020bures}. The following proposition shows that the $L_2$-Wasserstein distance between such measures with respect to a Mahalanobis cost metric (which we term as generalized Wasserstein distance) coincides with the GBW distance in \eqref{gbw_distance}. 

\begin{proposition}
\label{prop_wasserstein_dist}
Define the generalized Wasserstein distance as $\tilde{W}^2(\mu, \nu) := \inf_{\bx \sim \mu ,\by \sim \nu} \allowbreak \mathbb{E} \| \bx - \by\|^2_{\bM^{-1}}$, for any $\bM^{-1} \in \sS_{++}^n$. Suppose $\mu, \nu$ are two Gaussian measures with zero mean and covariances as $\bX, \bY \in \sS_{++}^n$ respectively. Then, we have $\tilde{W}^2(\mu, \nu) = d^2_{\rm gbw}(\bX, \bY)$.
\end{proposition}

Alternatively, the same distance is recovered by considering two scaled random Gaussian vector $\bM^{-1/2} \bx, \bM^{-1/2} \by$ under the Euclidean distance, i.e., $d^2(\bX, \bY) = \inf_{\bx \sim \mu, \by \sim \nu}\mathbb{E} \| \bM^{-1/2} \bx-\bM^{-1/2} \by \|^2_2$. For completeness, we also derive the optimal transport plan corresponding to the GBW distance in Appendix.

\section{Generalized Bures-Wasserstein Riemannian geometry}
In this section, the geometry arising from the GBW metric (\ref{gbw_metric}) is shown to have a Riemannian structure for a given $\bM \in \sS^{n}_{++}$, which we denote as $\M_{\rm gbw}$. We show the expressions of the Riemannian distance, geodesic, exponential/logarithm maps, Levi-Civita connection, sectional curvature as well as the geometric mean and barycenter. A summary of the results is presented in Table \ref{table_riem_geometry_compare}. Additionally, we discuss optimization on the SPD manifold with the proposed GBW geometry. We defer the detailed derivations discussed in this section to Appendix. 

% for the GBW Riemannian geometry in the rest of the section to the supplementary, where we only provide proofs from the first perspective of Riemannian submersion. The proofs take inspiration from the analysis in \cite{bhatia2019bures,malago2018wasserstein,massart2019curvature}. 

\subsection{Differential geometric properties of GBW}
\label{sec:bw_geometry}

To derive the various expressions in Table \ref{table_riem_geometry_compare}, we provide two strategies, one is by a Riemannian submersion from the general linear group and another is by a Riemannian isometry from the BW Riemannian geometry, $\M_{\rm bw}$. These claims are formalized in Propositions \ref{prop_riem_submersion} and \ref{prop_riem_isometry} respectively.

\paragraph{\textbf{Perspective from Riemannian submersion:}}
A \textit{Riemannian submersion} \cite{lee2006riemannian} between two manifolds is a smooth surjective map where its differential restricted to the horizontal space is isometric (formally defined in Appendix). The general linear group ${\rm GL}(n)$ is the set of invertible matrices with the group action of matrix multiplication. When endowed with the standard Euclidean inner product $\langle \cdot, \cdot\rangle_2$, the group becomes a Riemannian manifold, denoted as $\M_{\rm gl}$. The proposition below introduces a Riemannian submersion from $\M_{\rm gl}$ to $\M_{\rm gbw}$. 

\begin{table}[t]
\begin{center}
{\small 
\caption{Summary of expressions for the proposed generalized Bures-Wasserstein (GBW) Riemannian geometry, which is parameterized by $\bM \in \sS_{++}^d$. 
}
\label{table_riem_geometry_compare}
\begin{tabular}{@{}p{0.14\textwidth}p{0.85\textwidth}@{}}
\toprule
% & Generalized Bures-Wasserstein Geometry \\
% \midrule
Metric & $g_{\rm gbw}(\bU, \bV) =  \frac{1}{2}\trace( \L_{\bX,\bM}[\bU]\bV)$\\
\addlinespace[5pt]
Distance  & $d_{\rm gbw}^2(\bX, \bY) =  \trace(\bM^{-1}\bX) + \trace(\bM^{-1} \bY) - 2 \trace(\bX\bM^{-1}\bY\bM^{-1})^{1/2} $\\
\addlinespace[5pt]
Geodesic  & $\gamma(t) = ((1-t) \bX^{1/2} + t \bY^{1/2} \bO)((1-t) \bX^{1/2} + t \bY^{1/2} \bO)^\top$ with $\bO$ the orthogonal polar factor of $\bY^{1/2} \bM^{-1} \bX^{1/2}$.\\
\addlinespace[5pt]
Exp  &  ${\rm Exp}_{\bX} (\bU) = \bX + \bU +  \bM \L_{\bX, \bM}[\bU] \bX \L_{\bX, \bM}[\bU] \bM$\\
\addlinespace[5pt]
Log & ${\rm Log}_{\bX}(\bY) = \bM(\bM^{-1} \bX \bM^{-1} \bY)^{1/2} + (\bY \bM^{-1}\bX \bM^{-1})^{1/2}\bM - 2\bX$\\
\addlinespace[5pt]
Connection & $\nabla_\xi\eta = \D_\xi\eta + \{ \bX \L_{\bX, \bM}[\eta]\bM \L_{\bX,\bM}[\xi] \bM + \bX  \L_{\bX, \bM}[\xi]\bM \L_{\bX,\bM}[\eta] \bM  \}_{\rm S} - \{ \bM \L_{\bX, \bM}[\eta]\xi \}_{\rm S} - \{ \bM \L_{\bX, \bM}[\xi] \eta \}_{\rm S},$ where $\{\bA \}_{\rm S} \coloneqq \frac{1}{2}(\bA + \bA^\top)$.\\
\addlinespace[5pt]
Min/Max Curvature & $K_{\min}(\pi(\bP)) = 0, {\text{ and }} K_{\max}(\pi(\bP)) = \frac{3}{\sigma_n^2 + \sigma^2_{n-1}}$, where $\sigma_i$ is the $i$-th largest singular value of $\bP$, and $\pi(\bP) = \bM^{1/2} \bP \bP^\top \bM^{1/2}$. \\
\bottomrule
\end{tabular}
}
\end{center}
\end{table}

\begin{proposition}
\label{prop_riem_submersion}
The map $\pi: \M_{\rm gl} \xrightarrow{} \M_{\rm gbw}$ defined as $\pi(\bP) = \bM^{1/2} \bP \bP^\top \bM^{1/2}$ is a Riemannian submersion, for $\bP \in {\rm GL}(n)$ and $\M_{\rm gbw}$ parameterized by $\bM \in \sS_{++}^n$ as in \eqref{gbw_metric}. 
\end{proposition}

\paragraph{\textbf{Perspective from Riemannian isometry:}}
A \textit{Riemannian isometry} between two manifolds is a diffeomorphism (i.e., bijective, differentiable, and its inverse is differentiable) that pulls back the Riemannian metric from one to another \cite{lee2006riemannian}. We show in the following proposition that there exists a Riemannian isometry between the GBW and BW geometries.

\begin{proposition}
\label{prop_riem_isometry}
Define a map as $\tau(\bD) = \bM^{-1/2} \bD \bM^{-1/2}$, for $\bD \in \sS^n$. Then, the GBW metric can be written as $g_{\rm gbw, \bX}(\bU, \bV) = g_{\rm bw, \tau(\bX)}(\tau(\bU), \tau(\bV))$, where the subscript $\bX, \tau(\bX)$ indicates the tangent space. Hence, $\tau: \M_{\rm bw} \xrightarrow{} \M_{\rm gbw}$ is a Riemannian isometry.
\end{proposition}

The proofs of the results in Table \ref{table_riem_geometry_compare} are in Appendix and derived from the first perspective of Riemannian submersion, taking inspiration from the analysis in \cite{bhatia2019bures,malago2018wasserstein,massart2019curvature}. In Appendix, we also include various additional developments on the GBW geometry, such as geometric interpolation and barycenter, connection to robust Wasserstein distance and metric learning.

% \alertAH{shorten or remove?}

\subsection{Riemannian optimization with the GBW geometry}

% \section{Applications on Riemannian optimization}
\label{sec:applications}

% \alertBM{Take out sections 4.4 and 4.5 to experiments}

% In this section, we discuss a number of applications where the proposed generalized BW geometry may be employed.

% and provide the optimization ingredients for the same. 

% \subsection{Riemannian optimization with the GBW geometry}
\label{riemannian_opt_gbw}

Learning over SPD matrices usually concerns optimizing an objective function with respect to the parameter, which is constrained to be SPD. Riemannian optimization is an elegant approach that converts the constrained optimization into an unconstrained problem on manifolds~\cite{absil2008a,boumal2020a}. Among the metrics for the SPD matrices, the affine-invariant (AI) metric is seemingly the most popular choice for Riemannian optimization due to its efficiency and convergence guarantees. Recently, however, in~\cite{han2021riemannian}, the BW metric is shown to be a promising alternative for various learning problems. Below, we derive the expressions for Riemannian gradient and Hessian of an objective function for the GBW geometry. 

Riemannian gradient (and Hessian) are generalized gradient (and Hessian) on the tangent space of Riemannian manifolds. The expressions allow to implement various Riemannian optimization methods, using toolboxes like Manopt~\cite{boumal2014manopt}, Pymanopt~\cite{townsend16a}, ROPTLIB \cite{huang2018roptlib}, etc.
% and Manopt.jl \cite{Bergmann2022}, 

\begin{proposition}
\label{riem_grad_hess_gbw}
The Riemannian gradient and Hessian on $\M_{\rm gbw}$ is derived as $\grad f(\bX) = 2\bX \nabla f(\bX) \bM + 2\bM \nabla f(\bX) \bX$ and $\hess f(\bX)[\bU] =  4 \{ \bM \nabla^2f(\bX)[\bU] \bX \}_{\rm S} \allowbreak + 2\{ \bM \nabla f(\bX) \bU \}_{\rm S} + 4 \{ \bX \{ \nabla f(\bX) \bM \L_{\bX, \bM}[\bU] \}_{\rm S} \bM \}_{\rm S} -  \{ \bM \L_{\bX, \bM}[\bU] \grad f(\bX) \}_{\rm S}$,
% \begin{align*}
%     \hess f(\bX)[\bU] =\, &4 \{ \bM \nabla^2f(\bX)[\bU] \bX \}_{\rm S} + 2\{ \bM \nabla f(\bX) \bU \}_{\rm S} \\
%     &+ 4 \{ \bX \{ \nabla f(\bX) \bM \L_{\bX, \bM}[\bU] \}_{\rm S} \bM \}_{\rm S} - \{ \bM \L_{\bX, \bM}[\bU] \grad f(\bX) \}_{\rm S},
% \end{align*}
where $\nabla f(\bX), \nabla^2 f(\bX)$ represent the Euclidean gradient and Hessian, respectively.
\end{proposition}

In Appendix, we discuss {\it geodesic convexity} of functions on the SPD manifold endowed with the GBW metric. It generalizes the discussion in \cite{han2021riemannian}.

% Riemannian steepest descent extends gradient descent by updating the iterates as $x_{t+1} = {\rm Exp}_{x_t}(\eta\, \grad f(x_t))$, where $\eta$ is the stepsize. Riemannian trust region is a second-order method that ...

\section{Experiments}\label{sec:experiments}

In this section, we perform experiments showing the benefit of the GBW geometry. The algorithms are implemented in Matlab using the Manopt toolbox~\cite{boumal2014manopt}. The codes are available on \url{https://github.com/andyjm3/GBW}.

%\alertPJ{We need to improve this.}

\subsection{Log-determinant Riemannian optimization}\label{sec:log_det_optimization}

\paragraph{\textbf{Problem formulation:}} Log-determinant (log-det) optimization is common in statistical machine learning, such as for estimating the covariance, with an objective concerning $\min_{\bX \in \sS_{++}^n} f(\bX) = - \log\det(\bX)$. From \cite{han2021riemannian}, optimization with the BW geometry is less well-conditioned compared to the AI geometry. This is because the Riemannian Hessians at optimality are ${\rm Hess}_{\rm ai} f(\bX^*)[\bU] = \bU$ for the AI geometry and ${\rm Hess}_{\rm bw} f(\bX^*)[\bU] = 4 \{ (\bX^*)^{-1} \bU \}_{\rm S}$ for the BW geometry. This suggests, under the BW geometry, the condition number of Hessian at optimality depends on the solution $\bX^*$, while no dependence on $\bX^*$ under the AI geometry. Thus, this leads to a poor performance on BW geometry \cite{han2021riemannian}. 

Here, we show how the GBW geometry helps to address this issue. Specifically, with the GBW geometry, we see from Proposition \ref{riem_grad_hess_gbw} that by choosing $\bM = \bX^*$, the Riemannian Hessian is ${\rm Hess}_{\rm gbw} f(\bX^*)[\bU] = \bU$, which becomes well-conditioned (around the optimal solution). This provides the motivation for a choice of $\bM$. As the optimal solution $\bX^*$ is unknown in optimization problems, choice of $\bM$ is not trivial. In practice, one may choose $\bM = \bX$ dynamically at every or after a few iterations. This strategy corresponds to modifying the GBW geometry dynamically with iterations. 

\begin{table}[t]
\begin{center}
{\small 
\caption{Riemannian optimization ingredients for the affine-invariant (AI) and Generalized Bures-Wasserstein (GBW) with $\bM = \bX$ geometries for log-det optimization.}
\label{table_optimization_ingredient_AIGBW}
\begin{tabular}{@{}p{0.1\textwidth}p{0.42\textwidth}p{0.42\textwidth}@{}}
\toprule
 & AI & GBW (with $\bM = \bX$)\\
\midrule
Exp & ${\rm Exp}_{\bX} (\bU) = \bX \exp(\bX^{-1} \bU)$  &  ${\rm Exp}_{\bX} (\bU) = \bX + \bU + \frac{1}{4} \bU \bX^{-1} \bU$\\
\addlinespace[4pt]
Grad & $\grad f(\bX) = \bX\bC \bX -\bX$ & $\grad f(\bX) =  4 \bX \bC \bX - 4 \bX$\\
\addlinespace[4pt]
Hess & $\hess f(\bX)[\bU] = 2\bU + \{ \bU \bC \bX \}_{\rm S}$ & $\hess f(\bX)[\bU] = 2 \bU + 2\{ \bU \bC \bX\}_{\rm S} $\\
\bottomrule
\end{tabular}
}
\end{center}
\end{table}

As an example, we consider the following inverse covariance estimation problem \cite{friedman2008sparse,hsieh2011sparse} as $\min_{\bX \in \sS_{++}^n} f(\bX) = -\log\det(\bX) + \trace(\bC\bX),$ 
% \begin{equation}\label{eq:log_det_formulation}
% \min_{\bX \in \sS_{++}^n} f(\bX) = -\log\det(\bX) + \trace(\bC\bX),
% \end{equation}
where $\bC \in \sS_{++}^n$ is a given SPD matrix. The Euclidean gradient $\nabla f(\bX) = -\bX^{-1} + \bC$ and the Euclidean Hessian $\nabla^2 f(\bX)[\bU] = \bX^{-1} \bU \bX^{-1}$. From the analysis in Appendix,
this problem is geodesic convex and the optimal solution is $\bX^* = \bC^{-1}$, which we seek to estimate as a direct computation is challenging for ill-conditioned $\bC$. 
% \alertPJ{If optimal solution is known, why are we solving it?}

Choosing $\bM = \bX$ and following derivations in Section \ref{riemannian_opt_gbw}, the expressions for the exponential map, Riemannian gradient, and Hessian under the GBW geometry are shown in Table \ref{table_optimization_ingredient_AIGBW}, where we also draw comparisons to the AI geometry. We see that the choice of $\bM = \bX$ allows GBW to locally approximate the AI geometry up to some constants. 
% The gradient expressions are the same (up to a scaling which is immaterial). However, the differences appear in the Hessian and exponential map expressions.
For example, the AI exponential map $\bX \exp(\bX^{-1} \bU)$ can by approximated by second-order terms as $\bX + \bU + \frac{1}{2} \bU \bX^{-1} \bU$. This matches the GBW expression up to an additional term $ \frac{1}{4} \bU \bX^{-1} \bU$. 
% Similarly, there is a difference of the term $\{ \bU \bC \bX\}_{\rm S}$ in the Hessian expression. 
Overall, the similarity of optimization ingredients help GBW (with $\bM = \bX$) perform as similar as the AI geometry, which helps to resolve the poor performance of BW for log-det optimization problems observed in \cite{han2021riemannian}.

%\subsection{Results on log-determinant Riemannian optimization}

\begin{figure*}[t]
%\captionsetup{justification=left}
    \centering
    \subfigure[\texttt{\footnotesize Log-det:} \texttt{\footnotesize well-cond} ]{\includegraphics[scale=0.22]{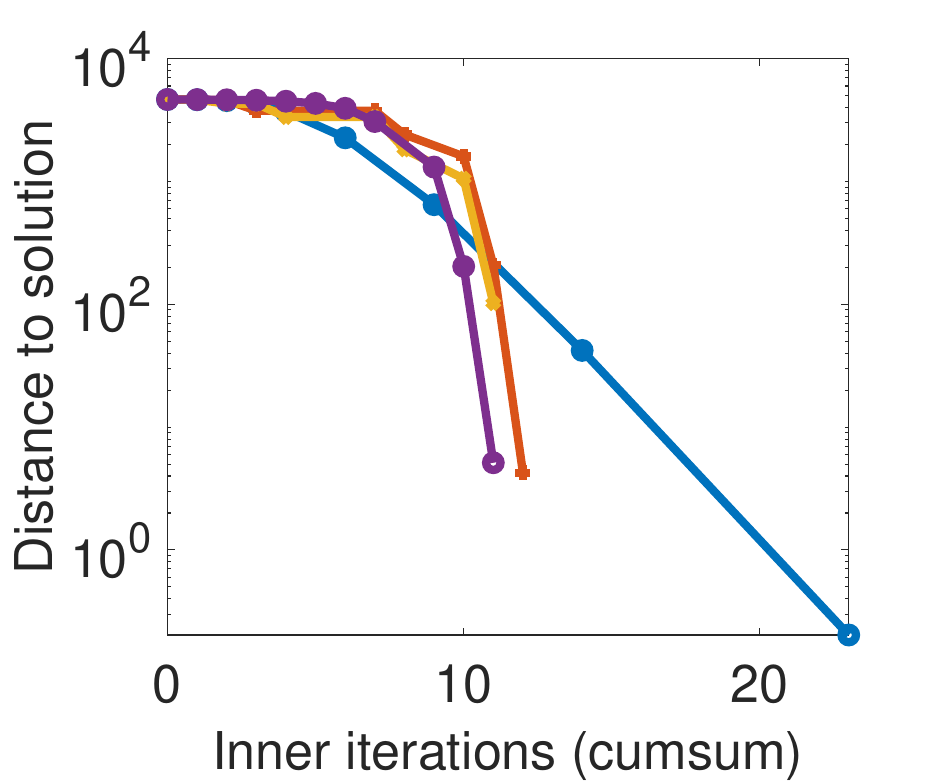}}
    \subfigure[\texttt{\footnotesize Log-det:} \texttt{\footnotesize ill-cond} ]{\includegraphics[scale=0.22]{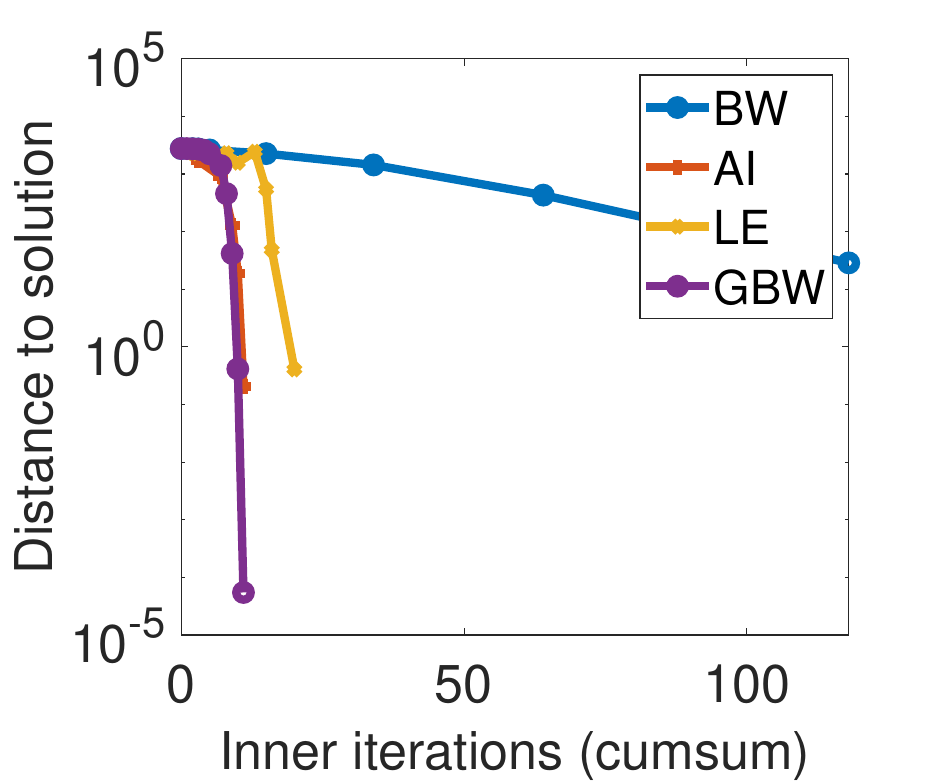}}
    \subfigure[\texttt{\footnotesize GMM:} \texttt{\footnotesize iris} ]{\includegraphics[scale=0.22]{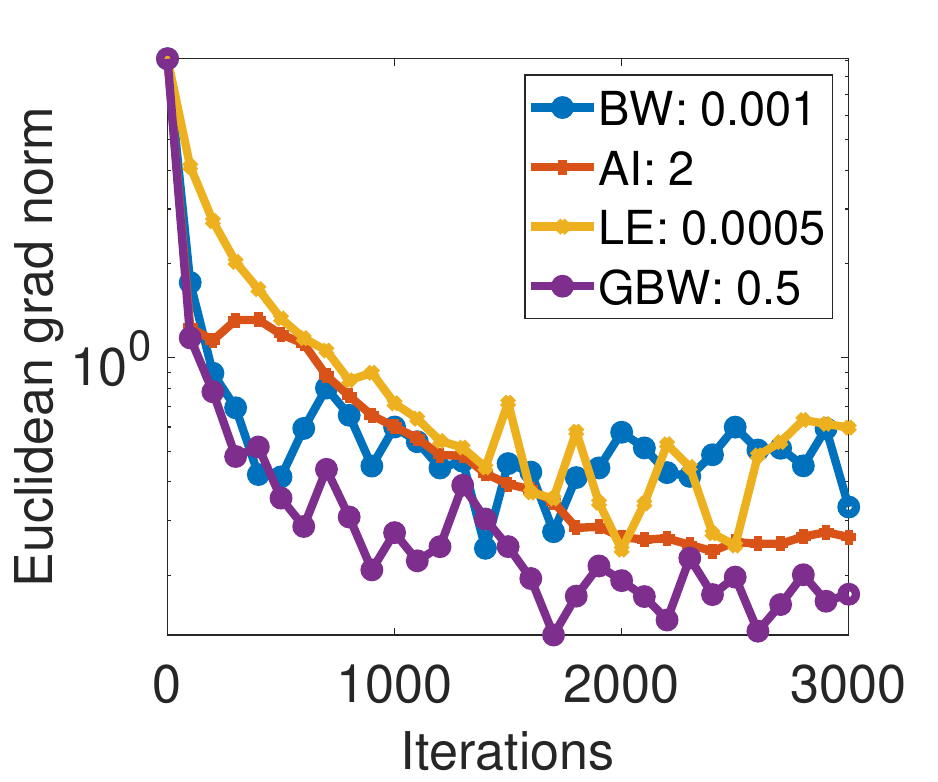}} \\[-0.12in]
    \subfigure[\texttt{\footnotesize GMM:} \texttt{\footnotesize kmeansdata} ]{\includegraphics[scale=0.22]{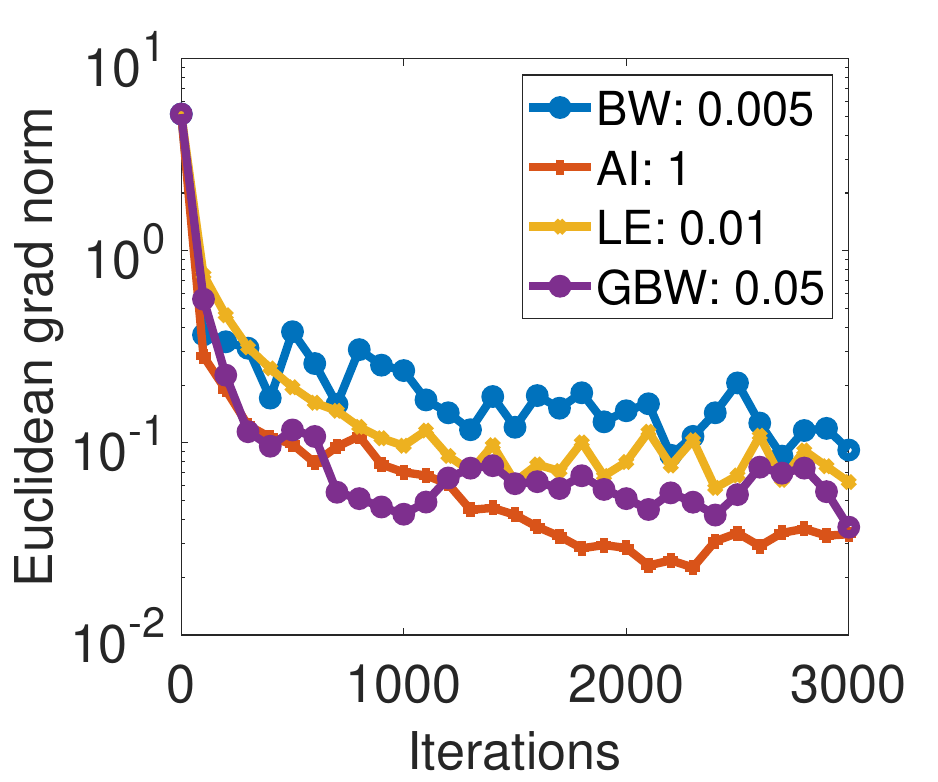}}
    \subfigure[\texttt{\footnotesize GMM:} \texttt{\footnotesize balance} ]{\includegraphics[scale=0.22]{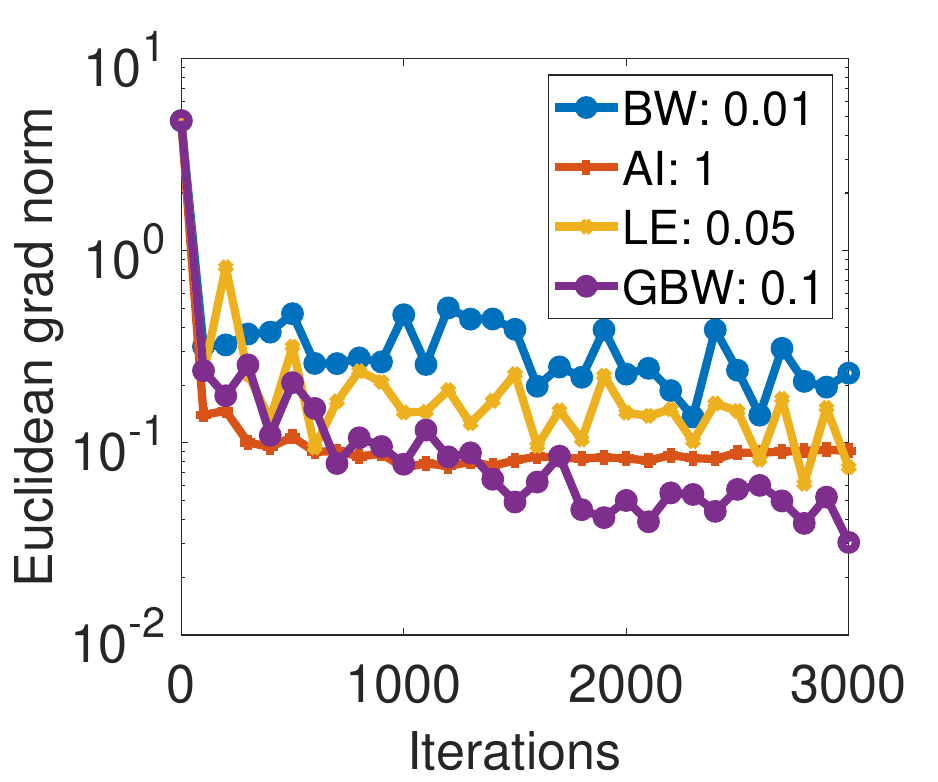}}
    \subfigure[\texttt{\footnotesize GMM:} \texttt{\footnotesize phoneme} ]{\includegraphics[scale=0.22]{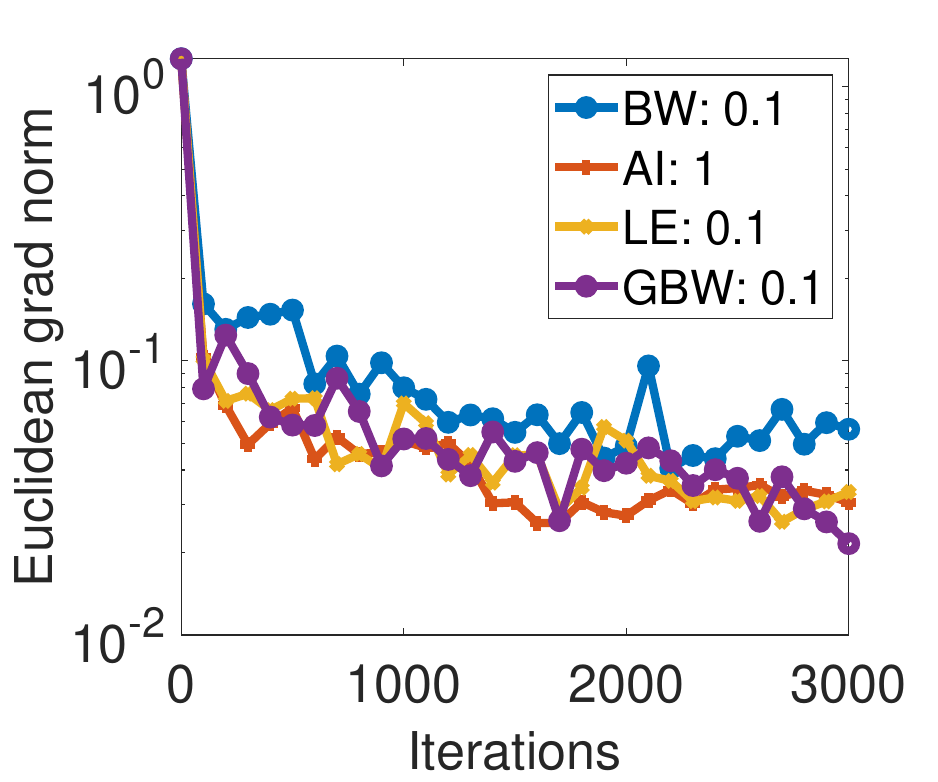}}
    \caption{Figures~(a)~\&~(b): Convergence for log-det optimization problem via Riemannian trust region algorithm. Figures~(c)-(f): Gaussian mixture model via Riemannian stochastic gradient descent algorithm with optimal initial stepsize. 
    In both the settings, the GBW algorithm outperforms the BW algorithm and performs similar to the AI algorithm. This can be attributed to the choice of $\bM$, which offers additional flexibility to the GBW modeling. 
    % Poor performance of the BW algorithm in such settings has also been reported in~\cite{han2021riemannian}.
    }
    \label{logdet_gmm_optimization_figure}
\end{figure*}

% \paragraph{\textbf{Tasks:}} 
% We show experiments on problem (\ref{eq:log_det_formulation}). 

\paragraph{\textbf{Experimental setup and results:}} We follow the same settings as in \cite{han2021riemannian} to create problem instances and consider two instances where the condition number of $\bX^*$ is $10$ (well-conditioned) and $1000$ (ill-conditioned). $\bC$ is then obtained as ${(\bX^*)}^{-1}$.
To compare the convergence performance of optimization methods under the AI, LE, BW, and GBW (with $\bM = \bX$) geometries, we implement the Riemannian trust region (a second-order solver) with the considered geometries \cite{absil2008a,boumal2020a}. To measure convergence, we use the distance to (theoretical) optimal solution, i.e., $\| \bX_t - \bX^* \|_{2}$. We plot this distance against the cumulative inner iterations that the trust region method takes to solve a particular trust region sub-problem at every iteration. The inner iterations are a good measure to show convergence of trust region algorithms \cite[Chapter~7]{absil2008a}.

% \paragraph{\textbf{Results:}} 
From Figures~\ref{logdet_gmm_optimization_figure}(a)~\&~\ref{logdet_gmm_optimization_figure}(b), we observe the faster convergence with the GBW geometry compared to other geometries regardless of the condition number. In contrast, the BW geometry performs poorly in log-determinant optimization problems as shown in~\cite{han2021riemannian}. The GBW geometry effectively resolves the convergence issues with the BW geometry for such settings. Based on our discussion earlier, we see that GBW with $\bM = \bX$ performs similar to the AI geometry. Empirically, it shows that the GBW geometry effectively bridges the gap between BW and AI geometries for optimization problems. 

\subsection{Gaussian mixture model (GMM)}
\paragraph{\textbf{Problem formulation:}} We now consider Gaussian density estimation and mixture model problem. 
%Another notable example of log-det optimization is the Gaussian density estimation and mixture model problem. 
Let $\bx_i \in \sR^d, i =1,...,N,$ be the given i.i.d. samples. 
Following~\cite{hosseini2020alternative}, we consider a reformulated GMM problem on augmented samples $\by_i^\top = [\bx_i^\top;1]\in \sR^{d+1}$. 
The density of a GMM is parameterized by the augmented covariance matrix $\bSigma \in \sR^{d+1}$. 
It should be noted that the log-likelihood of Gaussian is geodesic convex under the AI geometry~\cite{hosseini2020alternative} but not under the GBW geometry. However, if we define $\bS = \bSigma^{-1}$~\cite{han2021riemannian}, the reparameterized log-likelihood $p_{\gN}(\bY; \bS) = \sum_{i=1}^N \log \left( (2\pi)^{1 - d/2} \exp(1/2) \det(\bS)^{1/2} \allowbreak \exp(-\frac{1}{2} \by_i^\top \bS \by_i )\right)$ is geodesic convex on $\M_{\rm gbw}$. Similar trick was employed in~\cite{han2021riemannian} to obtained geodesic convex log-likelihood objective for GMM under the BW geometry. 
Overall, we solve the GMM problem similar as discussed in  \cite{hosseini2020alternative,han2021riemannian}.

%In addition, we also consider the problem of Gaussian mixture model (GMM) \cite{hosseini2020alternative}, and follow the same setting as in \cite{han2021riemannian}. From Proposition~\ref{prop_geodesic_convex}, it can be verified that the GMM objective is geodesic convex under the GBW geometry.  

% i.e.,
% \begin{equation}\label{eq:GMM}
%     \max_{\substack{\{\bSigma_j \in \sS_{++}^n\},  \{\omega_j : \sum_{j=1}^K \omega_j = 1\}}} L = \sum_{i=1}^N \log \Big( \sum_{j=1}^K \omega_j  p_\gN(\bx_i ; \bSigma_j) \Big),
% \end{equation}
% where $\bx_i \in \sR^n, i=1,\ldots,N$ are samples with $K$ Gaussian components, {{{with reformulated Gaussian density}}}
% $p_\gN(\bx ; \bSigma)\allowbreak = (2\pi)^{1-d/2}\det(\bSigma)^{1/2}\exp(\frac{1}{2} - \frac{1}{2} \bx^\top \bSigma \bx).$

\paragraph{\textbf{Experimental setup and results:}} We consider datasets:
% \footnote{Obtained from the sites  \url{https://au.mathworks.com/help/stats/sample-data-sets.html} and \url{https://sci2s.ugr.es/keel/datasets.php}.} 
\texttt{iris}, \texttt{kmeansdata}, \texttt{balance}, and \texttt{phoneme} from Matlab database and Keel database \cite{derrac2015keel}. For comparisons, we implement the Riemannian stochastic gradient descent method~\cite{bonnabel2013stochastic} as it is widely used in GMM problems~\cite{hosseini2020alternative}. The batch size is set to $50$ and we use a decaying stepsize for all the geometries~\cite{han2021riemannian}. As discussed in Section~\ref{sec:log_det_optimization}, we set $\bM = \bX$ at every iteration for optimizing under the GBW geometry. Without access to the optimal solution, the convergence is measured in terms of the Euclidean gradient norm $\| \bSigma_t \nabla L(\bSigma_t)\|_2$ for comparability across geometries. 

% \paragraph{\textbf{Results:}} 
Figures~\ref{logdet_gmm_optimization_figure}(c)-\ref{logdet_gmm_optimization_figure}(f) show convergence along with the best selected initial stepsize. We observe that convergence under the GBW geometry is competitive and clearly outperforms the BW geometry based algorithm.

%\subsection{Results on geometry-aware PCA}

\begin{remark}
For all the experiments in this section, we simply set $\bM = \bX$. In general, $\bM$ can be learned according to the applications. We demonstrate several examples in Appendix.
\end{remark}

\section{Conclusion}
In this paper, we propose a Riemannian geometry that generalizes the recently introduced Bures-Wasserstein geometry for SPD matrices. This generalized geometry has natural connections to the orthogonal Procrustes problem as well as to the optimal transport theory, and still possesses the properties of the Bures-Wasserstein geometry (which is a special case). The new geometry is shown to be parameterized by a SPD matrix $\bM$. This offers necessary flexibility in applications. Experiments show that learning of $\bM$ leads to better modeling in applications.

% For example, minimizing $\bM^{-1}$ (over a set) while minimizing the GBW distance is akin to doing principal components analysis.

% As future research, we intend to investigate how learning of $\bM$ improves convergence of algorithms and whether this leads to a principled preconditioning approach for problems with SPD matrices under the GBW geometry. 

% Existing works already formalize a family of metrics unifying Bures-Wasserstein and Log-Euclidean (LE) \cite{minh2019unified}, as well as affine-invariant (AI) and Log-Euclidean \cite{thanwerdas2019affine} metrics. An interesting research direction would be to consider unification of all three and whether the GBW generalization offers additional insights for this. This is especially relevant given our analysis for log-determinant optimization, where we find similarity of optimization ingredients for AI and GBW (with a particular choice of $\bM$).

% %\acks{Acknowledgements should go at the end, before appendices and references. You can uncomment this for the camera-ready version on paper acceptance.}

\bibliographystyle{plain}
% \bibliography{references}

\clearpage

\appendix

\section{Additional results and proofs for Section \protect{\ref{sec:motivations}}}

\subsection{Proof of Proposition \protect{\ref{gbw_procrustes}}}
\begin{proof}[Proof of Proposition \ref{gbw_procrustes}]
First we have 
\begin{align}
    &\min_{\bO \in O(n)} \| \bX^{1/2} - \bY^{1/2} \bO \|^2_{\bM^{-1}} \nonumber\\
    &= \trace(\bM^{-1} \bX) + \trace(\bM^{-1} \bY) - 2 \max_{\bO \in O(n)} \trace(\bM^{-1} \bX^{1/2} \bO \bY^{1/2}). \label{prop1_eq1}
\end{align}
And the minimum of \eqref{prop1_eq1} is attained when $\bO$ is the orthogonal polar factor of $\bY^{1/2}\bM^{-1}\bX^{1/2}$, which is $\bO = \bY^{1/2}\bM^{-1}\bX^{1/2}(\bX^{1/2}\bM^{-1}\bY\bM^{-1}\bX^{1/2})^{-1/2}$, as proved in \cite{bhatia2019bures}. Substituting the expression of $\bO$ in \eqref{prop1_eq1} completes the proof.
\end{proof}

\subsection{Proof of Proposition \protect{\ref{prop_wasserstein_dist}}}

Before we proceed to prove Proposition \ref{prop_wasserstein_dist}, we provide an essential lemma, which generalizes \cite[Theorem~2]{bhatia2019bures}.

\begin{lemma}
Define $\tilde{F}(\bX,\bY) = \trace(\bX^{1/2} \bM^{-1} \bY \bM^{-1} \bX^{1/2})^{1/2}$. Then for any $\bX, \bY \in \sS_{++}^n$, 
\begin{enumerate}[1)]
    \item $\tilde{F}(\bX, \bY) = \min_{\bA \in \sS_{++}^n} \frac{1}{2} \trace( \bX\bA + \bM^{-1}\bY \bM^{-1}\bA^{-1} )$.
    \item $\tilde{F}(\bX, \bY) = \min_{\bA \in \sS_{++}^n} \sqrt{\trace(\bX \bA) \trace(\bM^{-1}\bY \bM^{-1}\bA^{-1})}$. 
\end{enumerate}
\end{lemma}
\begin{proof}
Following \cite{bhatia2019bures}, the proof proceeds by analyzing the first-order stationary conditions, where we replace $\bY$ with $\bM^{-1} \bY \bM^{-1}$. 
\end{proof}

\begin{proof}[Proof of Proposition \protect{\ref{prop_wasserstein_dist}}]
We have $\bX = \mathbb{E}[\bx\bx^\top]$ and $\bY = \mathbb{E}[\by\by^\top]$. 
\begin{align*}
    \tilde{W}^2(\mu, \nu) &= \inf_{\bx \sim \mu, \by \sim \nu} \mathbb{E}[ \bx^\top \bM^{-1} \bx + \by^\top \bM^{-1} \by - 2\bx^\top \bM^{-1} \by ] \\
    &= \inf_{\bx \sim \mu, \by \sim \nu} \trace(\bM^{-1} \bX) + \trace(\bM^{-1} \bY) - 2\trace(\bM^{-1} \mathbb{E}[\by\bx^\top] ) ] \\
    &= \trace(\bM^{-1} \bX) + \trace(\bM^{-1} \bY) - \sup_{\bK: \bSigma \succeq \bzero} 2\trace(\bM^{-1} \bK^\top),
\end{align*}
where $\bK$ is the covariance between $\bx, \by$ such that the joint covariance matrix 
\begin{equation*}
    \bSigma =\mathbb{E} \begin{bmatrix} \bx\bx^\top & \bx\by^\top \\ \bx\by^\top & \by\by^\top \end{bmatrix} =  \begin{bmatrix} \bX & \bK \\ \bK^\top & \bY \end{bmatrix} \succeq \bzero.
\end{equation*}
Two necessary and sufficient conditions for $\bSigma \succeq \bzero$ are (i) $\bX \succeq \bK \bY^{-1} \bK^\top$ and (ii) $\bK = \bX^{1/2}\bC\bY^{1/2}$ for some contraction $\bC$, i.e., $\| \bC \|_2 \leq 1$ \cite{bhatia2009positive}. Hence, $\trace(\bK) \leq \| \bX^{1/2} \|_{2} \| \bY^{1/2} \|_{2} = \sqrt{\trace(\bX) \trace(\bY)}$. Also, for any $\bA \in \sS_{++}^n$, the block diagonal matrix 
\begin{equation*}
    \bP = \begin{bmatrix}\bA^{1/2} & \bzero \\ \bzero & \bA^{-1/2} \bM^{-1} \end{bmatrix} \in {\rm M}(2n)
\end{equation*}
Then
\begin{align*}
    \bP \begin{bmatrix} \bX & \bK \\ \bK^\top &\bY  \end{bmatrix}     \bP^\top = 
    \begin{bmatrix}
    \bA^{1/2} \bX \bA^{1/2} & \bA^{1/2} \bK \bM^{-1}\bA^{-1/2} \\ \bA^{-1/2} \bM^{-1} \bK^\top \bA^{1/2} &
\bA^{-1/2} \bM^{-1}\bY \bM^{-1} \bA^{-1/2}  \end{bmatrix} \succeq \bzero.
\end{align*}
This leads to 
\begin{align*}
    \trace(\bM^{-1} \bK^\top) &= \trace(\bA^{-1/2} \bM^{-1} \bK^\top \bA^{1/2}) \\
    &\leq \sqrt{\trace(\bA^{1/2} \bX \bA^{1/2}) \trace(\bA^{-1/2} \bM^{-1}\bY \bM^{-1} \bA^{-1/2})} \\
    &= \sqrt{\trace(\bX\bA) \trace(\bM^{-1}\bY \bM^{-1}\bA^{-1})}.
\end{align*}
Hence choosing $\bK = (\bX \bY)^{1/2}$, we can show $\bK \bY^{-1} \bK^\top = \bX$ \cite{bhatia2019bures} and
\begin{equation*}
    \max_{\bK:\bSigma \succeq \bzero} \trace(\bM^{-1} \bK^\top) = \tilde{F}(\bX, \bY) = \min_{\bA \in \sS_{++}^n}\sqrt{\trace(\bX\bA) \trace(\bM^{-1}\bY \bM^{-1}\bA^{-1})}.
\end{equation*}
It thus follows that $\tilde{W}^2(\mu, \nu) = d^2_{\rm gbw}(\bX, \bY)$.
\end{proof}

\subsection{Additional results: optimal transport plan for GBW distance}

Next, we derive the optimal transport plan corresponding to the generalized Wasserstein distance as follows, where we use the notation $\bA \# \bB$ to represent the matrix geometric mean under the affine-invariant metric, i.e., $\bA \# \bB = \bA^{1/2} (\bA^{-1/2}\bB \bA^{-1/2})^{1/2} \bA^{1/2} = \bA (\bA^{-1}\bB)^{1/2} = (\bA \bB^{-1})^{1/2} \bB$.

\begin{proposition}
\label{prop_ot_plan}
Let $\bx, \by \in \sR^n$ be random Gaussian vectors with zero mean and covariance matrices $\bX, \bY \in \sS_{++}^n$ respectively. The optimal transport plan from $\bx$ to $\by$ under the Mahalanobis distance is $\bT = \bM(\bX^{-1} \# (\bM^{-1} \bY \bM^{-1}))$.
\end{proposition}

\begin{proof}[Proof of Proposition \ref{prop_ot_plan}]
For any $\bP \in {\rm M}(n)$ as a transport plan,
\begin{align}
    &\mathbb{E}\| \bM^{-1/2}\bx - \bP \bM^{-1/2}\bx \|^2_2 \nonumber\\
    &= \trace(\bM^{-1}\bX) + \trace(\bP \bM^{-1/2} \bX \bM^{-1/2} \bP^\top ) - 2\trace(\bX^{1/2} \bM^{-1/2} \bP \bM^{-1/2} \bX^{1/2}). \label{main_prop_4}
\end{align}
By comparing \eqref{main_prop_4} to $d^2_{\rm gbw}(\bX, \bY)$, we set 
\begin{equation*}
    \bX^{1/2} \bM^{-1/2} \bP \bM^{-1/2} \bX^{1/2} = (\bX^{1/2} \bM^{-1} \bY \bM^{-1} \bX^{1/2})^{1/2}
\end{equation*}
which gives an expression of $\bP$ as 
\begin{align}
    \bP &= \bM^{1/2} \bX^{-1/2} (\bX^{1/2} \bM^{-1} \bY \bM^{-1} \bX^{1/2})^{1/2} \bX^{-1/2} \bM^{1/2} \nonumber\\
    &= \bM^{1/2} \big( \bX^{-1}\#(\bM^{-1} \bY \bM^{-1}) \big) \bM^{1/2}. \label{transport_P}
\end{align}
From this result, we have
\begin{align*}
    &\trace( \bP \bM^{-1/2} \bX \bM^{-1/2} \bP^\top) \\
    &= \trace \Big( \bM^{1/2} \big( \bX^{-1}\#(\bM^{-1} \bY \bM^{-1}) \big) \bX \big( \bX^{-1}\#(\bM^{-1} \bY \bM^{-1}) \big) \bM^{1/2} \Big) \\
    &= \trace \Big( \bM^{1/2} \bX^{-1} (\bX \bM^{-1} \bY \bM^{-1})^{1/2} \bX \bX^{-1} (\bX \bM^{-1} \bY \bM^{-1})^{1/2} \bM^{1/2} \Big) \\
    &= \trace(\bM^{-1} \bY),
\end{align*}
where we use the property of matrix geometric mean. This suggests the definition of $\bP$ is the optimal transport map under the Euclidean distance. Combining \eqref{transport_P} with \eqref{main_prop_4} shows
\begin{align*}
    &\mathbb{E}\| \bM^{-1/2} \bx - \bM^{1/2} \big( \bX^{-1}\#(\bM^{-1} \bY \bM^{-1}) \big) \bx \|^2_2 \\
    = \,&\mathbb{E} \| \bx - \bM \big( \bX^{-1} \# (\bM^{-1} \bY \bM^{-1}) \big) \bx \|_{\bM^{-1}}^2  \\
    = \,&d^2_{\rm gbw}(\bX, \bY).
\end{align*}
We can, thus, define the transport plan as $\bT_{\bX \xrightarrow{} \bY} := \bM \big( \bX^{-1} \# (\bM^{-1} \bY \bM^{-1}) \big)$ and denote $\by = \bT_{\bX \xrightarrow{} \bY} \,\bx$, which is a Gaussian random vector with covariance 
\begin{equation*}
    \bT_{\bX \xrightarrow{} \bY} \bX \bT_{\bX \xrightarrow{} \bY}^\top = \bM \big( \bX^{-1} \# (\bM^{-1} \bY \bM^{-1}) \big) \bX \big( \bX^{-1} \# (\bM^{-1} \bY \bM^{-1}) \big) \bM = \bY.
\end{equation*}
Thus, $\mathbb{E}\| \bx - \by \|^2_{\bM^{-1}} = d^2_{\rm gbw}(\bX, \bY)$ where $\by = \bT_{\bX \xrightarrow{} \bY} \bx$. From Proposition \ref{prop_wasserstein_dist}, we see $\bT_{\bX \xrightarrow{} \bY}$ is the {optimal transport plan} from $\bX$ to $\bY$ under the Mahalanobis distance.
\end{proof}

\section{Additional results and proofs for Section \protect{\ref{sec:bw_geometry}}}

\subsection{Riemannian distance, geodesics, exponential map, and logarithm map}
% We derive the Riemannian distance, geodesic, exponential map and logarithm map on $\M_{\rm gbw}$ in the following propositions. 

% We first derive the Riemannian distance and geodesic expressions in the following propositions.
\begin{proposition}
\label{dist_prop}
The Riemannian distance on $\M_{\rm gbw}$ is derived as $d_{\rm gbw}(\bX, \bY) = ( \trace(\bM^{-1}\bX) + \trace(\bM^{-1} \bY) - 2 \trace(\bX\bM^{-1}\bY\bM^{-1})^{1/2})^{1/2}$.
\end{proposition}
% We have shown the connection of the GBW Riemannian distance and the orthogonal Procrustes problem and Wasserstein distance with a generalized Mahalanobis metric. We also highlight that such claim extends beyond the two discussed settings. For example, one can verify the GBW distance corresponds to a generalized \textit{Hellinger distance} \cite{hellinger1909neue} for measuring similarity between probability vectors where $\bX, \bY$ are diagonal matrices. 

\begin{proposition}
\label{proposition_geodesic}
A geodesic on $\M_{\rm gbw}$ between any $\bX, \bY \in \sS_{++}^n$ is given by $\gamma(t) = (\pi \circ c)(t)$, where $c(t) = (1-t)\bM^{-1/2}\bX^{1/2} + t\bM^{-1/2}\bY^{1/2}\bO$. Here, $\bO$ is the orthogonal polar factor of $\bY^{1/2}\bM^{-1}\bX^{1/2}$.  
\end{proposition}
The geodesic in Proposition \ref{proposition_geodesic} can be simplified as 
\begin{align}
    \gamma(t) = \psi(t) \psi(t)^\top =  \big((1-t)\bX^{1/2} + t\bY^{1/2} \bO \big) \big( (1-t)\bX^{1/2} + t\bY^{1/2} \bO \big)^\top, \label{geodesic_simplify_bw}
\end{align}
which coincides with the geodesic of the BW geometry except that $\bO$ is now the orthogonal polar factor of $\bY^{1/2} \bM^{-1} \bX^{1/2}$ rather than $\bY^{1/2} \bX^{1/2}$ as for BW. 
% We can also rewrite the geodesic for GBW as 
% \begin{align*}
%     &\gamma(t) = (1-t)^2 \bX + t^2 \bY + t(1-t) \left( \bY^{1/2} \bO \bX^{1/2} + \bX^{1/2} \bO^\top \bY^{1/2} \right) \\
%     &= (1-t)^2\bX + t^2 \bY + t(1-t) \left( \bY\#(\bM\bX^{-1}\bM) \bM^{-1}\bX +\bX \bM^{-1} (\bY \# \bM\bX^{-1}\bM) \right) \\
%     &= (1-t)^2\bX + t^2 \bY + t(1-t) \left(  (\bY \bM^{-1}\bX \bM^{-1})^{1/2}\bM + \bM (\bM^{-1}\bX \bM^{-1} \bY)^{1/2} \right).
% \end{align*}

\begin{proposition}
\label{prop_exp_log_map}
The Riemannian exponential map associated with the generalized BW metric is ${\rm Exp}_{\bX}(t\bU) = \bX + t\bU + t^2 \bM \L_{\bX, \bM}[\bU] \bX \L_{\bX, \bM}[\bU] \bM$. The neighbourhood $\mathcal{X} := \{ \bM + t \bM \L_{\bX, \bM}[\bU] \bM \in \sS_{++}^n  \}$ is a totally normal neighbourhood where exponential map is a diffeomorphism with logarithm map ${\rm Log}_{\bX}(\bY) = \bM(\bM^{-1} \bX \bM^{-1} \bY)^{1/2} + (\bY \bM^{-1}\bX \bM^{-1})^{1/2}\bM - 2\bX$. 
\end{proposition}
We remark that the exponential map is only invertible in the neighbourhood $\mathcal{X}$, where $t$ is chosen sufficiently small. This makes $\M_{\rm gbw}$ a geodesic incomplete manifold, similar to $\M_{\rm bw}$ \cite{malago2018wasserstein}.

\subsection{Levi-Civita connection and sectional curvature}
% Here, we show the expressions for the Levi-Civita connection and sectional curvature of $\M_{\rm gbw}$. The formal definitions are provided as supplementary. 

The Levi-Civita connection (Levi-Civita derivative) of a vector field on manifold $\M$ is the unique covariant derivative that satisfies (1) torsion-free property, and (2) metric compatibility (formal definitions are provided as supplementary). Let $\mathfrak{X}(\M)$ be the space of vector fields on the Riemannian manifold $(\M, g)$ and denote $\{ \bA\}_{\rm S} := (\bA + \bA)/2$, for $\bA \in \sR^{n \times n}$.

\begin{proposition}
\label{riem_connection}
The Levi-Civita connection with the GBW geometry is $\nabla_\xi \eta = \D_\xi\eta + \{ \bX \L_{\bX, \bM}[\eta]\bM \L_{\bX,\bM}[\xi] \bM + \bX  \L_{\bX, \bM}[\xi]\bM \L_{\bX,\bM}[\eta] \bM  \}_{\rm S} - \{ \bM \L_{\bX, \bM}[\eta]\xi \}_{\rm S} - \{ \bM \L_{\bX, \bM}[\xi] \eta \}_{\rm S}$, for any $\xi, \eta \in \mathfrak{X}(\M)$.
\end{proposition}

Sectional curvature measures the curvature locally around a point $\bX$, which is defined geometrically as the Gaussian curvature of a $2$-dimensional subspace of $T_{\bX}\M_{\rm gbw}$.

\begin{proposition}
\label{prop_sect_curv}
Let $U, V \in \mathfrak{X}(\M_{\rm gbw})$ be two (linearly independent) vector fields. Let $\tilde{U}(\bP) = \bM^{1/2} \L_{\bM, \pi(\bP)}[U(\pi(\bP))] \bM^{1/2} \bP$, for $\bP \in \M_{\rm gl}$ and similarly for $\tilde{V}$. Suppose $\tilde{U}(\bP)$, $\tilde{V}(\bP)$ are orthonormal on $T_\bP \M_{\rm gl}$. Then, the sectional curvature of the subspace spanned by $U(\pi(\bP)),\allowbreak V(\pi(\bP))$ is
\begin{equation*}
    K(U(\pi(\bP)), V(\pi(\bP))) = \sum_{i,j} \frac{3 \bC_{ij}^2}{\sigma^2_j(\sigma_i \sigma_j^{-1} + \sigma_i^{-1} \sigma_j)^2},
\end{equation*}
where $\bC = \bV^\top (\tilde{V}(\bP)^\top \tilde{U}(\bP) - \tilde{U}(\bP)^\top \tilde{V}(\bP)) \bV$ and the singular value decomposition gives $\bP = \bU \bSigma \bV^\top$ with $\bSigma = {\rm diag}(\sigma_1, \ldots, \sigma_n)$, $\sigma_1 \geq \cdots  \geq \sigma_n$.
\end{proposition}

The sectional curvature of GBW geometry is shown to be non-negative, regardless of the choice of $\bM$. The bounds are shown below.

\begin{proposition}
\label{prop_bounds_sect_curvature}
Under the same settings as Proposition \ref{prop_sect_curv}, the minimum sectional curvature is zero and the maximum is $3/(\sigma_n^2 + \sigma_{n-1}^2)$.
\end{proposition}

Although sharing the minimum curvature with the BW geometry \cite{massart2019curvature}, the maximum curvature of GBW geometry is affected by the choice of $\bM$ from the definition of Riemannian submersion $\pi$ defined in Proposition \ref{prop_riem_submersion}.

\subsection{Geometric interpolation and barycenter}

With the proposed GBW geometry, we are interested to study the properties of interpolation between two or more SPD matrices on the manifold. This has implications in various applications, such as diffusion tensor imaging (DTI) \cite{pennec2006riemannian,bhatia2009positive}. 

First, we show that the geodesic interpolation between $\bX, \bY \in \M_{\rm gbw}$ satisfies an operator inequality (shown in supplementary), which is $\gamma(t) \preceq (1-t) \bX + t \bY$, where $\gamma(0) = \bX$, $\gamma(1) = \bY$ and $\preceq$ denotes the L\"owner partial order. One immediate implication is $\log\det(\gamma(t)) \leq \log\det((1-t) \bX + t\bY)$ that suggests a smaller swelling effect compared to the Euclidean interpolation for DTI applications. See more discussions in the supplementary. 

We further study the interpolation for multiple SPD matrices $\{\bX_{l}\}_{l = 1}^N$, also known as the barycenter problem on $\M_{\rm gbw}$, i.e., $\min_{\bA \in \sS_{++}^n } \sum_{l=1}^N w_l d^2_{\rm gbw}(\bX_l, \bA)$, where the weights $\sum_l w_l = 1$. We show in the supplementary that there exists a unique solution to the problem and provide a fixed point iteration to compute the barycenter \cite{bhatia2019bures}.

\subsection{Proof of Proposition \protect{\ref{prop_riem_submersion}}}

First we recall a \textit{smooth submersion} is a smooth map $\pi: (\M, g) \xrightarrow{} (\gN, h)$ from Riemannian manifold $(\M, g)$ to $(\gN,h)$ such that its differential $\D \pi(x): T_x\M \xrightarrow{} T_{\pi(x)}\gN$ is surjective for any $x \in \M$. Every tangent space $T_x\M$ can be decomposed as $T_x\M = \gV_x \oplus \gH_x = \ker(\D\pi(x)) \oplus \ker(\D\pi(x))^\perp$, where $\ker(f)$ denotes the kernel of a map and $\oplus$ is the direct sum. We respectively call $\gV_x, \gH_x$ as the vertical and horizontal subspaces. The map $\pi$ is called a \textit{Riemannian submersion} if it is a smooth submersion and its differential restricted to the horizontal space, $\D\pi(x): \gH_x \xrightarrow{} T_{\pi(x)}\gN$ is isometric for any $x \in \M$, i.e. $g_x(u,v) = h_{\pi(x)}( \D\pi(x)[u], \D\pi(x)[v] )$.

\begin{proof}[Proof of Proposition \ref{prop_riem_submersion}]
Note that the tangent space of $\M_{\rm gl}$, $T_\bP \M_{\rm gl}$, is the space of $\sR^{n \times n}$. 
The differential of $\pi(\bP)$ in the direction $\bU \in \sR^{n \times n}$ is given by $\D\pi(\bP)[\bU] = \bM^{1/2} \bU \bP^\top \bM^{1/2} + \bM^{1/2} \bP\bU^\top \bM^{1/2}$. We then derive the kernel of $\D\pi(\bP)$ (vertical space $\gV_{\bP}$) and the orthogonal complement of the kernel (horizontal space $\gH_{\bP}$) as 
\begin{align}
    \ker(\D\pi(\bP)) &= \{ \bU: \D\pi(\bP)[\bU] = \bzero \} \nonumber\\
    &= \{ \bU = \bM^{-1/2} \bK \bM^{-1/2} \bP^{-\top}: \bK \text{ is skew-symmetric}\}, \label{vertical_space}\\
    \ker(\D\pi(\bP))^\perp &= \{ \bV : \trace(\bV^\top\bM^{-1/2} \bK \bM^{-1/2} \bP^{-\top} ) = \bzero \} \nonumber\\
    &= \{ \bV = \bM^{1/2} \bS \bM^{1/2} \bP : \bS \in \sS^n \}. \label{horizontal_space}
\end{align}
It is clear that $\pi$ is a smooth submersion. Now, we only need to verify that it also satisfies the isometry property. For any $\bS, \bH \in \sS^n$, $\bM^{1/2} \bS \bM^{1/2} \bP, \bM^{1/2} \bH \bM^{1/2} \bP \in \gH_\bP$, and
\begin{align*}
    \D \pi(\bP)[\bM^{1/2} \bS \bM^{1/2} \bP] &= \bM \bS \bM^{1/2} \bP\bP^\top \bM^{1/2} + \bM^{1/2} \bP\bP^\top \bM^{1/2} \bS \bM \\
    &= \bM \bS \pi(\bP) + \pi(\bP) \bS \bM \\
    \D \pi(\bP)[\bM^{1/2} \bH \bM^{1/2} \bP] &= \bM \bH \bM^{1/2} \bP\bP^\top \bM^{1/2} + \bM^{1/2} \bP\bP^\top \bM^{1/2} \bH \bM \\
    &= \bM \bH \pi(\bP) + \pi(\bP) \bH \bM.
\end{align*}
The inner product at $\pi(\bP)$ is given by 
\begin{align*}
    &\langle \D \pi(\bP)[\bM^{1/2} \bS \bM^{1/2} \bP], \D \pi(\bP)[\bM^{1/2} \bH \bM^{1/2} \bP]\rangle_{\rm gbw} \\
    = &\frac{1}{2}\trace(\L_{\pi(\bP), \bM}[\D \pi(\bP)[\bM^{1/2} \bS \bM^{1/2} \bP]] \D \pi(\bP)[\bM^{1/2} \bH \bM^{1/2} \bP]) \\
    = &\frac{1}{2} \trace(\bS \bM \bH \pi(\bP) + \bS \pi(\bP)\bH \bM ) = \trace(\pi(\bP)\bS \bM \bH),
\end{align*}
where the last equality is because $\bS, \bM, \bH, \pi(\bP)$ are all symmetric. The inner product at $\bP$ is given by 
\begin{align*}
    \langle \bM^{1/2} \bS \bM^{1/2} \bP, \bM^{1/2} \bH \bM^{1/2} \bP\rangle_{2} = \trace(\bP^\top \bM^{1/2}\bS\bM \bH \bM^{1/2}\bP) = \trace(\pi(\bP) \bS\bM \bH).
\end{align*}
This shows for any $\tilde{\bS}, \tilde{\bH} \in \gH_\bP$, $\langle \tilde{\bS}, \tilde{\bH} \rangle_{2} = \langle \D\pi(\bP)[\tilde{\bS}], \D\pi(\bP)[\tilde{\bH}] \rangle_{\rm gbw}$, thereby completing the proof.
\end{proof}

\subsection{Proof of Proposition \protect{\ref{prop_riem_isometry}}}

\begin{proof}[Proof of Proposition \ref{prop_riem_isometry}]
Given for any $\bX \in \sS_{++}^n$, $\tau: \sS_{++}^n \xrightarrow{} \sS_{++}^n$ is a diffeomorphism, it is thus suffices to show $g_{\rm gbw, \bX}(\bU, \bV) = g_{\rm bw, \tau(\bX)}(\tau(\bU), \tau(\bV))$. That is,
\begin{align*}
    g_{\rm gbw, \bX}(\bU,\bV) &= \frac{1}{2} \trace( \L_{\bX, \bM}[\bU] \bV ) = \frac{1}{2} \trace( \bM^{-1/2} \L_{\tau(\bX)}[\tau(\bU)] \bM^{-1/2} \bV) \\
    &=  \frac{1}{2} \trace( \L_{\tau(\bX)}[\tau(\bU)] \tau(\bV)) = g_{\rm bw, \tau(\bX)}(\tau(\bU), \tau(\bV)),
\end{align*}
where we use the definition of the Lyapunov operator.
\end{proof}

\subsection{Proof of Proposition \protect{\ref{dist_prop}}}

First we provide a Theorem that shows the pushforward distance from a Riemannian submersion is the Riemannian distance.

\begin{theorem}[Riemannian distance induced from Riemannian submersion  \cite{van2020bures}]
\label{R_dist_theorem}
Consider $\pi: (\M, g) \xrightarrow{} (\gN,h)$ as a Riemannian submersion. Let $d_\M$ be the Riemannian distance on $(\M, g)$ and the pushforward distance $d_\gN(p,q) = \inf_{u \in \pi^{-1}(p), v \in \pi^{-1}(q)} d_\M(u,v)$ is equal to the Riemannian distance.
\end{theorem}

We now proceed to derive the distance expression.

\begin{proof}[Proof of Proposition \ref{dist_prop}]
From the definition of $\pi$ and Theorem \ref{R_dist_theorem}, we have for any $\bX, \bY \in \sS^n_{++}$, 
\begin{align*}
    d_{\rm gbw}^2(\bX, \bY) &= \inf_{\bOmega, \bR \in O(n)} d_{\rm gl}^2(\bM^{-1/2}\bX^{1/2}\bOmega, \bM^{-1/2}\bY^{1/2}\bR)\\
    &= \inf_{\bOmega, \bR \in O(n)} \|\bM^{-1/2}\bX^{1/2}\bOmega -  \bM^{-1/2}\bY^{1/2}\bR \|^2_{2} \\
    &= \trace(\bM^{-1}\bX) + \trace(\bM^{-1} \bY) - 2 \sup_{\bOmega, \bR \in O(n)} \trace(\bM^{-1} \bX^{1/2} \bOmega \bR^\top \bY^{1/2}) \\
    &= \trace(\bM^{-1}\bX) + \trace(\bM^{-1} \bY) - 2 \sup_{\bO \in O(n)} \trace(\bM^{-1} \bX^{1/2} \bO \bY^{1/2}).
\end{align*}
The supremum is attained when $\bO = \bY^{1/2}\bM^{-1}\bX^{1/2}(\bX^{1/2}\bM^{-1}\bY\bM^{-1}\bX^{1/2})^{-1/2}$ as in Proposition \ref{gbw_procrustes}. This completes the proof.
\end{proof}

Here we also verify that the second-order approximation of the GBW distance recovers the proposed Riemannian metric in \eqref{gbw_metric}.

\begin{proposition}
\label{prop_second_approximation_GBW}
The GBW distance is approximated as $d^2_{\rm gbw}(\bX, \bX+\theta \bH) = \frac{\theta^2}{2} \trace(\L_{\bX,\bM}[\bH] \bH) + o(\theta^2).$
\end{proposition}

\begin{proof}[Proof of Proposition \ref{prop_second_approximation_GBW}]
 For $\bX \in \sS^n_{++}$ and $\bH \in \sS^n$ such that $\bX \pm \bH \in \sS_{++}^n$. Thus, for $\theta \in [-1, 1]$, $\bX +\theta \bH \in \sS^n_{++}$ and
 \begin{align*}
     d^2_{\rm gbw}(\bX, \bX + \theta \bH) = &\,2\trace(\bM^{-1}\bX) + \theta \trace(\bM^{-1}\bH) \\
     &- 2\trace(\bX^{1/2} \bM^{-1}\bX \bM^{-1} \bX^{1/2} + \theta \bX^{1/2} \bM^{-1} \bH \bM^{-1} \bX^{1/2})^{1/2}
 \end{align*}
The first-order derivative is 
\begin{align*}
    &\frac{d}{d\theta} d^2_{\rm gbw}(\bX, \bX+ \theta \bH) \\
    &= \trace \Big(\bM^{-1} \bH \\
    &\quad- 2\L_{(\bX^{1/2}\bM^{-1}\bX \bM^{-1}\bX^{1/2} + \theta \bX^{1/2} \bM^{-1} \bH \bM^{-1} \bX^{1/2} )^{1/2}}[\bX^{1/2} \bM^{-1} \bH \bM^{-1}\bX^{1/2}] \Big) \\
    &= \trace \Big(\bM^{-1} \bH \\
    &\quad- (\bX^{1/2}\bM^{-1}\bX \bM^{-1}\bX^{1/2} + \theta \bX^{1/2} \bM^{-1} \bH \bM^{-1} \bX^{1/2} )^{-1/2} \bX^{1/2} \bM^{-1} \bH \bM^{-1}\bX^{1/2} \Big),
\end{align*}
where we use the properties of standard Lyapunov operator, $\D_{\bV}(\bX)^{1/2} = \L_{\bX^{1/2}}[\bV]$ and $\trace(\L_{\bX}[\bU]) = \frac{1}{2}\trace(\bX^{-1}\bU)$. Notice that 
\begin{align*}
    &\frac{d}{d\theta} d^2_{\rm gbw}(\bX, \bX+ \theta \bH)|_{\theta = 0} \\
    &= \trace(\bM^{-1}\bH) - \trace \Big(\bM^{-1}\bX^{1/2} (\bX^{1/2} \bM^{-1} \bX \bM^{-1} \bX^{1/2})^{-1/2} \bX^{1/2} \bM^{-1} \bH \Big) \\
    &= \trace(\bM^{-1}\bH) - \trace\Big( (\bM^{-1}\bX \bM^{-1} \bX)^{-1/2} \bM^{-1} \bX \bM^{-1} \bH \Big) = 0,
\end{align*}
where the second equality is from \eqref{Lemma_1_inter}. The second order derivative is 
\begin{align*}
    &\frac{d^2}{d\theta^2}  d^2_{\rm gbw}(\bX, \bX+ \theta \bH)\\
    &= - \trace \Big( \frac{d}{d\theta} \big(\bX^{1/2}\bM^{-1}\bX \bM^{-1}\bX^{1/2} + \theta \bX^{1/2} \bM^{-1} \bH \bM^{-1} \bX^{1/2} \big)^{-1/2} \times \\
    &\qquad\qquad \qquad\bX^{1/2} \bM^{-1} \bH \bM^{-1}\bX^{1/2} \Big) \\
    &= \trace( \frac{d}{d\theta} (-C^{-1/2}) \bX^{1/2} \bM^{-1} \bH \bM^{-1}\bX^{1/2} ),
\end{align*}
where we let $\bC = \bX^{1/2}\bM^{-1}\bX \bM^{-1}\bX^{1/2} + \theta \bX^{1/2} \bM^{-1} \bH \bM^{-1} \bX^{1/2}$. Then,
\begin{align*}
    \frac{d}{d\theta} (-\bC^{-1/2}) = \bC^{-1/2} \L_{C^{1/2}}[\bX^{1/2}\bM^{-1} \bH \bM^{-1} \bX^{1/2}] \bC^{-1/2}.
\end{align*}
Thus,
\begin{align*}
    &\frac{d^2}{d\theta^2} d^2_{\rm gbw}(\bX, \bX+\theta \bH)|_{\theta =0} \\
    = &\trace( (\bX^{1/2}\bM^{-1}\bX \bM^{-1}\bX^{1/2})^{-1/2} \L_{(\bX^{1/2}\bM^{-1}\bX \bM^{-1}\bX^{1/2})^{1/2}}[\bX^{1/2}\bM^{-1}\bH \bM^{-1}\bX^{1/2}] \times \\ &(\bX^{1/2}\bM^{-1}\bX \bM^{-1}\bX^{1/2})^{-1/2} \bX^{1/2}\bM^{-1}\bH \bM^{-1}\bX^{1/2} ) \\
    = &\trace( \bX^{-1/2} \L_{(\bX^{1/2}\bM^{-1}\bX \bM^{-1}\bX^{1/2})^{1/2}}[\bX^{1/2}\bM^{-1}\bH \bM^{-1}\bX^{1/2}] \bX^{-1/2} \bH ).
\end{align*}
Notice, similarly from \eqref{Lemma_1_inter},
\begin{align*}
    (\bX^{1/2}\bM^{-1}\bX \bM^{-1}\bX^{1/2})^{1/2} \bX^{-1/2}\bM = \bX^{-1/2}\bM(\bX^{-1}\bM \bX^{-1} \bM)^{-1/2} &= \bX^{1/2}, \text{ and } \\
    \bM\bX^{-1/2}(\bX^{1/2}\bM^{-1}\bX \bM^{-1}\bX^{1/2})^{1/2} &= \bX^{1/2}
\end{align*}
Let $\bL := \L_{(\bX^{1/2}\bM^{-1}\bX \bM^{-1}\bX^{1/2})^{1/2}}[\bX^{1/2}\bM^{-1}\bH \bM^{-1}\bX^{1/2}]$. Then,
\begin{align*}
\bH &= \bM\bX^{-1/2}\bL (\bX^{1/2}\bM^{-1}\bX \bM^{-1}\bX^{1/2})^{1/2} \bX^{-1/2}\bM \\
&\quad+ \bM \bX^{-1/2}(\bX^{1/2}\bM^{-1}\bX \bM^{-1}\bX^{1/2})^{1/2} \bL\bX^{-1/2} \bM \\
&= \bM\bX^{-1/2} \bL \bX^{1/2} + \bX^{1/2} \bL \bX^{-1/2}\bM \\
&= \bM\bX^{-1/2} \bL \bX^{-1/2} \bX + \bX \bX^{-1/2} \bL \bX^{-1/2} \bM.
\end{align*}
Thus, $\L_{\bX,\bM}[\bH] = \bX^{-1/2}\bL\bX^{-1/2}$ and $\frac{d^2}{d\theta^2} d^2_{\rm gbw}(\bX, \bX+\theta \bH)|_{\theta =0} = \trace(\L_{\bX,\bM}[\bH] \bH)$. This completes the proof.
\end{proof}

\subsection{An important lemma regarding the polar factor}
The next lemma studies the various expressions of the polar factor $\bO$, which is used throughout the proofs in the rest of the paper. 

\begin{lemma}
\label{lemma_U}
Consider $\bO$ as defined in the proof of Proposition \eqref{dist_prop}, then 
\begin{equation*}
    \bO = \bY^{1/2} (\bY^{-1} \bM \bX^{-1}\bM)^{1/2} \bM^{-1} \bX^{1/2} = \bY^{-1/2} (\bY \# (\bM\bX^{-1}\bM)) \bM^{-1} \bX^{1/2}.
\end{equation*}
\end{lemma}
\begin{proof}
From the definition of $\bO$,
  \begin{align}
     \bO &= \bY^{1/2}\bM^{-1}\bX^{1/2}(\bX^{1/2}\bM^{-1}\bY\bM^{-1}\bX^{1/2})^{-1/2} \nonumber\\
     &=\bY^{1/2}\bM^{-1}\bX^{1/2} (\bX^{1/2}\bM^{-1}\bY\bM^{-1}\bX^{1/2} )^{-1/2} \bX^{-1/2} \bM \bM^{-1} \bX^{1/2} \nonumber\\
     &= \bY^{1/2} (\bM^{-1} \bX \bM^{-1} \bY )^{-1/2} \bM^{-1} \bX^{1/2} \label{Lemma_1_inter}\\
     &= \bY^{1/2} (\bY^{-1} \bM \bX^{-1}\bM)^{1/2} \bM^{-1} \bX^{1/2}, \nonumber\\
     &= \bY^{-1/2} \bY \# (\bM\bX^{-1}\bM) \bM^{-1} \bX^{1/2}, \label{important_inter_lemma}
\end{align}
where \eqref{Lemma_1_inter} is proved as follows. Denote $\bC = (\bX^{1/2}\bM^{-1}\bY\bM^{-1}\bX^{1/2} )^{-1/2}$ and we have
\begin{align*}
     \bI &= \bC \bX^{1/2}\bM^{-1}\bY\bM^{-1}\bX^{1/2} \bC \\
     &= (\bM^{-1} \bX^{1/2} \bC \bX^{-1/2 }\bM) \bM^{-1} \bX \bM^{-1} \bY (\bM^{-1} \bX^{1/2} \bC \bX^{-1/2} \bM). 
\end{align*}
Thus, $\bM^{-1} \bX^{1/2} \bC \bX^{-1/2} \bM = (\bM^{-1} \bX \bM^{-1} \bY )^{-1/2}$.
\end{proof}

\subsection{Poof of Proposition \protect{\ref{proposition_geodesic}}}

To derive the geodesic expression, we need the following well-known theorem. 
\begin{theorem}[Geodesic induced from Riemannian submersion \protect{\cite{bhatia2019bures,lee2018introduction}}]
\label{theorem_geodesic}
Consider $\pi: (\M, g) \xrightarrow{} (\gN,h)$ as a Riemannian submersion. Let $c$ be a geodesic on $(\M, g)$ with $c'(0)$ is horizontal. Then, we have
\begin{enumerate}[(\arabic*)]
    \item $c'(t)$ is horizontal for all $t$. 
    \item $\gamma := \pi \circ c$ is a geodesic on $(\gN,h)$ of the same length as $c$.
\end{enumerate}
\end{theorem}

\begin{proof}[Proof of Proposition \ref{proposition_geodesic}]
First, we see $\gamma(0) = \bX, \gamma(1) = \bY$ and for $\bM, \bX, \bY \in \sS^n_{++}$,
\begin{align*}
    c(t) &= ((1-t) \bI + t \bM^{-1/2} \bY^{1/2}\bO \bX^{-1/2}\bM^{1/2} ) \bM^{-1/2}\bX^{1/2} \\
    % &= ((1-t) \bI + t \bM^{-1/2} \bY^{1/2}\bU \bX^{-1/2}\bM \bM^{-1/2} ) \bM^{-1/2}\bX^{1/2} \\
    % &= ((1-t) \bI + t \bM^{-1/2} \bY (\bY^{-1} \bM \bX^{-1}\bM)^{1/2} \bM^{-1/2} ) \bM^{-1/2}\bX^{1/2} \\
    &= ((1-t) \bI + t \bM^{-1/2} \bY\#(\bM\bX^{-1}\bM) \bM^{-1/2} ) \bM^{-1/2}\bX^{1/2},
\end{align*}
where the second equality follows from Lemma \ref{lemma_U}. It is clear that 
\begin{equation*}
    \bM^{-1/2} \bY\#(\bM\bX^{-1}\bM) \bM^{-1/2} \in \sS^{n}_{++},
\end{equation*}
and hence, $c(t)$ lies entirely in ${\rm GL}(n)$ for $t \in [0,1]$ as it is closed under matrix multiplication. Also, $c(t)$ is a line segment, and thus, it is a valid geodesic on $\M_{\rm gl}$. Now, we need to show $c'(0)$ is horizontal. Indeed, we have 
\begin{align*}
    c'(0) &= \bM^{-1/2} \bY^{1/2}\bO - \bM^{-1/2}\bX^{1/2} \\
    &= \bM^{1/2}(\bM^{-1}\bY^{1/2} \bO - \bM^{-1}\bX^{1/2}) \\
    &= \bM^{1/2}(\bM^{-1}\bY^{1/2}\bO \bX^{-1/2}\bM - \bI) \bM^{-1} \bX^{1/2} \\
    &= \bM^{1/2} (\bM^{-1} \bY (\bY^{-1} \bM \bX^{-1}\bM)^{1/2} \bM^{-1} - \bM^{-1}) \bM^{1/2} \bM^{-1/2} \bX^{1/2} \\
    &= \bM^{1/2} \bH \bM^{1/2} \bM^{-1/2} \bX^{1/2},
\end{align*}
where $\bH := \bM^{-1} \bY \#(\bM\bX^{-1}\bM) \bM^{-1} - \bM^{-1} \in \sS^n$. Thus, from the definition of the horizontal space in \eqref{horizontal_space}, we have $c'(0) \in \gH_{\bM^{-1/2}\bX^{1/2}}$. This completes the proof. In addition, from Theorem \ref{theorem_geodesic}, we verify that the square of the Riemannian distance $d^2_{\rm gbw}$ is the same as the straight-line distance on $\M_{\rm gl}$, which is $\| \bM^{-1/2} \bX^{1/2} -\bM^{-1/2} \bY^{1/2}\bO \|^2_{2} = \trace(\bM^{-1}\bX) + \trace(\bM^{-1}\bY) - 2 \trace(\bX^{1/2}\bM^{-1}\bY\bM^{-1}\bX^{1/2})^{1/2}$.
\end{proof}

\subsection{Proof of Proposition \protect\ref{prop_exp_log_map}}

\begin{proof}[Proof of Proposition \ref{prop_exp_log_map}]
We first simplify $(1-t) \bX^{1/2} + t\bY^{1/2} \bO = ((1-t) \bI + t\bY^{1/2}\bU \bX^{-1/2}) \bX^{1/2} = ((1-t) \bM + t\bY \# (\bM\bX^{-1}\bM) )\bM^{-1}\bX^{1/2}$. With $\bK := \bY \# (\bM\bX^{-1}\bM)$, we rewrite the geodesic as 
\begin{align*}
    \gamma(t) &= ( (1-t) \bX^{1/2} + t \bY^{1/2} \bO ) ((1-t)\bX^{1/2} + t \bY^{1/2} \bO )^\top \\
    &= ( (1-t) \bM + t \bK )\bM^{-1} \bX \bM^{-1}  ( (1-t) \bM + t \bK ) \\
    &= \bX + t \bX(\bM^{-1}\bK -\bI) + t (\bK \bM^{-1} - \bI) \bX \\
    &\quad+ t^2 \bM (\bM^{-1} \bK \bM^{-1} - \bM^{-1}) \bX  (\bM^{-1} \bK \bM^{-1} - \bM^{-1}) \bM.
\end{align*}
The first-order derivative is 
\begin{align*}
    \gamma'(0) &= (\bK -\bM) \bM^{-1} \bX + \bX \bM^{-1}(\bK -\bM) = (\bK\bM^{-1} - \bI)\bX + \bX (\bM^{-1}\bK - \bI) \\
    &= \bM (\bM^{-1} \bK \bM^{-1} - \bM^{-1}) \bX + \bX (\bM^{-1} \bK \bM^{-1} - \bM^{-1}) \bM.
\end{align*}
Hence, $\gamma(t)= \bX + t \gamma'(0) + t^2 \bM \L_{\bX, \bM}[\gamma'(0)] \bX \L_{\bX, \bM}[\gamma'(0)] \bM$. The exponential map, therefore, is 
\begin{align*}
    {\rm Exp}_{\bX}(t\bU) &= \bX + t\bU + t^2 \bM \L_{\bX, \bM}[\bU] \bX \L_{\bX, \bM}[\bU] \bM \\
    &= (\bI + t\bM \L_{\bX,\bM}[\bU]) \bX (\bI + t\L_{\bX, \bM}[\bU] \bM) \\
    &= (\bM + t\bM \L_{\bX, \bM}[\bU] \bM)\bM^{-1} \bX \bM^{-1} (\bM + t\bM \L_{\bX, \bM}[\bU]\bM).
\end{align*}
Note that ${\rm Exp}_{\bX}(t\bU) \in \sS_{++}^n$ if $\bM + t\bM \L_{\bX, \bM}[\bU] \bM \in \sS_{++}^n$.

To derive the logarithm map, let $\bY = {\rm Exp}_{\bX}(\bU)$. We first have 
\begin{align*}
    &\bM + \bM\L_{\bX, \bM}[\bU]\bM \\
    &= (\bM^{-1} \bX \bM^{-1})^{-1/2} \left( (\bM^{-1} \bX \bM^{-1})^{1/2} \bY (\bM^{-1} \bX \bM^{-1})^{1/2} \right)^{1/2} (\bM^{-1} \bX \bM^{-1})^{-1/2}.
\end{align*}
and 
\begin{align*}
    &\L_{\bX, \bM}[\bU] = - \bM^{-1} + \bM^{-1}(\bM^{-1} \bX \bM^{-1})^{-1/2} \times \\
    &\qquad \qquad\qquad 
    \left( (\bM^{-1} \bX \bM^{-1})^{1/2} \bY (\bM^{-1} \bX \bM^{-1})^{1/2} \right)^{1/2} 
    (\bM^{-1} \bX \bM^{-1})^{-1/2} \bM^{-1} 
\end{align*}
Hence, let $\bS_\bM := \left( (\bM^{-1} \bX \bM^{-1})^{1/2} \bY (\bM^{-1} \bX \bM^{-1})^{1/2} \right)^{1/2}$. Then,
\begin{align*}
    \bU &= \bX \L_{\bX, \bM}[\bU] \bM + \bM \L_{\bX, \bM}[\bU] \bX \\
    &= 2\{\bX\bM^{-1} (\bM^{-1} \bX \bM^{-1})^{-1/2} \bS_\bM (\bM^{-1} \bX \bM^{-1})^{-1/2}\}_{\rm S} - 2\bX\\
    % &+(\bM^{-1} \bX \bM^{-1})^{-1/2} \bS_\bM (\bM^{-1} \bX \bM^{-1})^{-1/2} \bM^{-1} \bX \\
    &= 2\{\bM (\bM^{-1} \bX \bM^{-1})^{1/2} \bS_\bM (\bM^{-1} \bX \bM^{-1})^{-1/2}\}_{\rm S} - 2\bX \\
    % &+ (\bM^{-1} \bX \bM^{-1})^{-1/2} \left( (\bM^{-1} \bX \bM^{-1})^{1/2} \bY (\bM^{-1} \bX \bM^{-1})^{1/2} \right)^{1/2} (\bM^{-1} \bX \bM^{-1})^{1/2} \bM - 2\bX \\
    &= \bM(\bM^{-1} \bX \bM^{-1} \bY)^{1/2} + (\bY \bM^{-1}\bX \bM^{-1})^{1/2}\bM - 2\bX,
\end{align*}
where we denote $\{\bA\}_{\rm S} := (\bA +\bA^\top)/2$, for $\bA \in \sR^{n \times n}$. This completes the proof.
\end{proof}

\subsection{Proof of Proposition \protect\ref{riem_connection}}

The Levi-Civita connection (or Levi-Civita derivative) of a vector field on a manifold $\M$ is the unique covariant derivative that satisfies (1) torsion-free property, i.e., $\nabla_\xi \eta - \nabla_\eta \xi = \D_\xi \eta - \D_\eta \xi = [\xi, \eta]$ and (2) metric compatibility, i.e., $\nabla_\xi \langle \eta, \xi \rangle_\M = \langle \nabla_\xi \eta, \zeta\rangle_\M + \langle \eta, \nabla_\xi \zeta \rangle_\M$, for any vector fields $\xi, \eta, \zeta$.

\begin{proof}[Proof of Proposition \ref{riem_connection}]
The Levi-Civita connection is derived by applying \cite[MD.3]{lang2012differential}. For any vector fields $\xi, \eta, \zeta$ on $\M_{\rm gbw}$, it satisfies for any $\bX \in \M_{\rm gbw}$,
\begin{align}
    &\langle \nabla_\xi \eta, \L_{\bX, \bM}[\zeta] \rangle_2 \nonumber\\
    = \, &\langle \D_\xi \eta, \L_{\bX, \bM}[\zeta] \rangle_2 + \frac{1}{2} \langle \eta, \D_\xi \L_{\bX, \bM}[\zeta] \rangle_2 + \frac{1}{2} \langle \xi, \D_{\eta} \L_{\bX, \bM}[\zeta] \rangle_2 - \frac{1}{2}\langle \xi, \D_{\zeta} \L_{\bX, \bM} [\eta] \rangle \nonumber\\
    = \, &\langle \D_\xi \eta, \L_{\bX, \bM}[\zeta] \rangle_2 +\frac{1}{2}\langle \xi, \L_{\bX, \bM}\left[\zeta \L_{\bX,\bM}[\eta] \bM + \bM \L_{\bX,\bM}[\eta] \zeta \right] \rangle_2  \nonumber\\
    &-\frac{1}{2} \langle \eta, \L_{\bX, \bM}\left[\xi \L_{\bX,\bM}[\zeta] \bM + \bM \L_{\bX,\bM}[\zeta] \xi \right] \rangle_2 \nonumber\\
    &- \frac{1}{2} \langle \xi , \L_{\bX, \bM}\left[\eta \L_{\bX,\bM}[\zeta] \bM + \bM \L_{\bX,\bM}[\zeta] \eta \right] \rangle_2. \label{eq_1_levi_civita}
\end{align}
The second term of \eqref{eq_1_levi_civita} is rewritten as
\begin{align}
    &\frac{1}{2}\langle \xi, \L_{\bX, \bM}\left[\zeta \L_{\bX,\bM}[\eta] \bM + \bM \L_{\bX,\bM}[\eta] \zeta \right] \rangle_2 \nonumber\\
    = \, &\frac{1}{2}\langle \L_{\bX, \bM}[\eta] \bM \L_{\bX, \bM}[\xi] + \L_{\bX, \bM}[\xi] \bM \L_{\bX, \bM}[\eta] , \zeta \rangle_2 \nonumber\\
    = \, &\langle \{ \L_{\bX, \bM}[\eta] \bM \L_{\bX, \bM}[\xi] \}_{\rm S} , \zeta\rangle_2 \nonumber\\
    = \,&\langle \L_{\bX, \bM}\large[ \bX \{ \L_{\bX, \bM}[\eta] \bM \L_{\bX, \bM}[\xi] \}_{\rm S} \bM + \bM  \{ \L_{\bX, \bM}[\eta] \bM \L_{\bX, \bM}[\xi] \}_{\rm S} \bX  \large], \zeta\rangle_2 \nonumber\\
    = \, &\langle \bX \{ \L_{\bX, \bM}[\eta] \bM \L_{\bX, \bM}[\xi] \}_{\rm S} \bM + \bM  \{ \L_{\bX, \bM}[\eta] \bM \L_{\bX, \bM}[\xi] \}_{\rm S} \bX , \L_{\bX, \bM}[\zeta] \rangle_2 \nonumber\\
    = \, &\langle \{ \bX \L_{\bX, \bM}[\eta]\bM \L_{\bX,\bM}[\xi] \bM + \bX  \L_{\bX, \bM}[\xi]\bM \L_{\bX,\bM}[\eta] \bM  \}_{\rm S} , \L_{\bX, \bM}[\zeta]\rangle_2. \label{temp_1_hess}
\end{align}
Similarly,
\begin{align}
    \frac{1}{2} \langle \eta, \L_{\bX, \bM}\left[\xi \L_{\bX,\bM}[\zeta] \bM + \bM \L_{\bX,\bM}[\zeta] \xi \right] \rangle_2 &= \langle \{ \bM \L_{\bX, \bM}[\eta]\xi \}_{\rm S} , \L_{\bX, \bM}[\zeta] \rangle_2 \label{temp_2_hess}\\
    \frac{1}{2} \langle \xi , \L_{\bX, \bM}\left[\eta \L_{\bX,\bM}[\zeta] \bM + \bM \L_{\bX,\bM}[\zeta] \eta \right] \rangle_2 &= \langle \{ \bM \L_{\bX, \bM}[\xi] \eta \}_{\rm S}, \L_{\bX, \bM}[\zeta] \rangle_2. \label{temp_3_hess}
\end{align}
Applying the results in \eqref{temp_1_hess}, \eqref{temp_2_hess}, and \eqref{temp_3_hess} in \eqref{eq_1_levi_civita}, the proof is complete.
\end{proof}

\subsection{Proof of Proposition \protect\ref{prop_sect_curv}}

We first provide the formal definition of sectional curvature. The curvature tensor $R$ is defined for any $X, Y, Z \in \mathfrak{X}(\M)$, $R(X, Y)Z := \nabla_X \nabla_Y Z - \nabla_Y \nabla_X Z - \nabla_{[X,Y]}Z$, where $[X, Y] = XY - YX$ is the Lie bracket and $\nabla$ is the Levi-Civita connection. At point $p$, $R_p$ defines a $(1,3)$-tensor on $T_p\M$ with $R_p(X(p), Y(p)) Z(p) \in T_p\M$, where $X(p), Y(p), Z(p) \in T_p\M$ are the vector fields evaluated at $p \in \M$.  The sectional curvature is the scalar curvature of a $2$-dimensional subspace of $T_p\M$, given by 
\begin{equation}
    K(u, v) = \frac{g(R_p(u, v)v, u)}{g(u, u) g(v,v) - (g(u,v))^2} \label{sect_cur_def}
\end{equation}
for $u, v \in T_p\M$ as two linearly independent tangent vectors that span the subspace.

Before deriving the sectional curvature, we require the following theorem from Riemannian submersion.

We now derive the sectional curvature of $\M_{\rm gbw}$ based on the following theorem from Riemannian submersion \cite{besse2007einstein,o1966fundamental}. 

\begin{theorem}
\label{curvature_theorem}
Let $\pi: (\tilde{\M}, \tilde{g}) \xrightarrow{} (\M, g)$ be a Riemannian submersion and consider $X, Y$ as smooth vector fields on $\M$. The horizontal lift $\tilde{X}$, $\tilde{Y}$ are unique vector fields on $\tilde{\M}$ such that $\tilde{X}(p), \tilde{Y}(p) \in \gH_p$ and $\D \pi(p)[\tilde{X}(p)] = X({\pi(p)}), \D \pi(p)[\tilde{Y}(p)] = Y({\pi(p)})$ for all $p \in \tilde{\M}$. Then, the sectional curvature is 
\begin{equation*}
    K(X, Y) = \tilde{K}(\tilde{X},\tilde{Y}) + \frac{3}{4} \frac{\| [ \tilde{X}, \tilde{Y} ]^{\gV} \|^2}{Q(\tilde{X}, \tilde{Y})},
\end{equation*}
where $Q(\tilde{X}, \tilde{Y}) = \tilde{g}(\tilde{X}, \tilde{X}) \tilde{g}(\tilde{Y}, \tilde{Y}) - (\tilde{g}(\tilde{X}, \tilde{Y}))^2$. $Z^{\gV}$ is the vertical component of a vector field and $\tilde{K}$ is the sectional curvature of $(\tilde{\M}, \tilde{g})$.
\end{theorem}

Directly from Theorem \ref{curvature_theorem} above, we see that the sectional curvature of $\M_{\rm gbw}$ is non-negative, given that $\M_{\rm gl}$ endowed with the flat Euclidean metric has zero curvature. Before we derive the sectional curvature, we need the following lemma to show projection to the horizontal/vertical space on $T_\bP\M_{\rm gl}$.

\begin{lemma}
\label{lemma_projection_horizontal}
Any $\bU \in T_\bP\M_{\rm gl}$ can be projected onto the vertical and horizontal spaces defined in Proposition \ref{prop_riem_submersion}, i.e., $\bU = \bU^{\gV} + \bU^{\gH}$, where
\begin{align*}
    \bU^{\gH} &= \bM^{1/2} \L_{\bM, \bM^{1/2}\bP\bP^\top \bM^{1/2}}[\bM^{1/2}(\bU \bP^\top + \bP \bU^\top)\bM^{1/2}] \bM^{1/2} \bP, \\
    \bU^\gV &= \bM^{-1/2} \L_{\bM^{-1}, (\bM^{1/2} \bP \bP^\top \bM^{1/2})^{-1}}[\bM^{-1/2}(\bU \bP^{-1} - \bP^{-\top} \bU^\top) \bM^{-1/2}] \bM^{-1/2} \bP^{-\top}.
\end{align*}
\end{lemma}
\begin{proof}
Based on Proposition \ref{prop_riem_submersion}, for $\bU \in T_\bP \M_{\rm gl}$, it can be decomposed as $\bU = \bU^{\gV} + \bU^{\gH} = \bM^{-1/2}\bK \bM^{-1/2} \bP^{-\top} + \bM^{1/2} \bS \bM^{1/2} \bP$, for $\bK$ skew-symmetric and $\bS$ symmetric. From the decomposition, $\bU^\top = - \bP^{-1} \bM^{-1/2} \bK \bM^{-1/2} + \bP^\top \bM^{1/2} \bS \bM^{1/2}$. Thus, we have 
\begin{equation*}
    \bM^{1/2}(\bU \bP^\top + \bP \bU^\top)\bM^{1/2} = \bM \bS \bM^{1/2} \bP \bP^\top \bM^{1/2} + \bM^{1/2} \bP \bP^\top \bM^{1/2} \bS \bM.
\end{equation*}
Hence, $\bS = \L_{\bM, \bM^{1/2}\bP\bP^\top \bM^{1/2}}[\bM^{1/2}(\bU \bP^\top + \bP \bU^\top) \bM^{1/2}]$. Similarly, we also have
\begin{align*}
    &\bM^{-1/2}(\bU \bP^{-1} - \bP^{-\top}\bU^\top) \bM^{-1/2} \\
    &= \bM^{-1} \bK \bM^{-1/2} \bP^{-\top} \bP^{-1} \bM^{-1/2} + \bM^{-1/2}\bP^{-\top} \bP^{-1} \bM^{-1/2}\bK \bM^{-1}
\end{align*}
Thus, $\bK = \L_{\bM^{-1}, (\bM^{1/2} \bP \bP^\top \bM^{1/2})^{-1}}[\bM^{-1/2} (\bU \bP^{-1} - \bP^{-\top} \bU^\top) \bM^{-1/2}]$, which is clearly skew-symmetric given that $\bU \bP^{-1} - \bP^{-\top} \bU^\top$ is skew-symmetric.
\end{proof}

Finally, we proceed to prove the main proposition. 
\begin{proof}[Proof of Proposition \ref{prop_sect_curv}]
We denote $\bS_U := \L_{\bM, \pi(\bP)}[U(\pi(\bP))]$ and similarly for $\bS_V$. Hence we see $\tilde{U}, \tilde{V}$ are the horizontal lift according to the definition. 

To start, it is clear $\tilde{U}(\bP) \in \gH_\bP$ according to the definition of the horizontal space in Proposition \ref{prop_riem_submersion}. Also, we have 
\begin{align*}
    &\D \pi(\bP)[\tilde{U}(\bP)] \\
    &= \bM  \L_{\bM, \pi(\bP)}[U(\pi(\bP))] \bM^{1/2} \bP \bP^\top \bM^{1/2} + \bM^{1/2} \bP \bP^\top \bM^{1/2} \L_{\bM, \pi(\bP)}[U(\pi(\bP))] \bM \\
    &= U(\pi(\bP)), \qquad \forall \, \bP \in \M_{\rm gl}.
\end{align*}

This suggests $\tilde{U} \in \mathfrak{X}(\M_{\rm gl})$ is indeed a horizontal lift of $U \in \mathfrak{X}(\M_{\rm gbw})$. Next we compute the sectional curvature following Theorem \ref{curvature_theorem}.

First, we derive an expression for the Lie bracket. For any two horizontal tangent vectors $\tilde{U}(\bP), \tilde{V}(\bP)$, they can be written as $\tilde{U}(\bP) = \bM^{1/2} \bS_U \bM^{1/2} \bP$ and $\tilde{V}(\bP) = \bM^{1/2} \bS_V \bM^{1/2} \bP$, for arbitrary symmetric matrices $\bS_U, \bS_V$. Therefore,
\begin{align*}
    &[\tilde{U}, \tilde{V}](\bP) \\
    &= \D \tilde{V}(\bP)[\tilde{U}(\bP)] - \D \tilde{U}(\bP)[\tilde{V}(\bP)] \\
    &= \bM^{1/2} \D \bS_V [\tilde{U}(\bP)] \bM^{1/2} \bP + \bM^{1/2} \bS_V \bM^{1/2} \tilde{U}(\bP) - \bM^{1/2} \D \bS_U [\tilde{V}(\bP)] \bM^{1/2} \bP \\
    &\quad- \bM^{1/2} \bS_U \bM^{1/2} \tilde{V}(\bP).
\end{align*}
From Lemma \ref{lemma_projection_horizontal}, to project the result onto the vertical space, we need to first evaluate
\begin{align*}
    &\bM^{-1/2} \big( ([\tilde{U}, \tilde{V}](\bP)) \bP^{-1} - \bP^{-\top}([\tilde{U}, \tilde{V}](\bP))^\top \big) \bM^{-1/2} \\
    = \, & \D\bS_V[\tilde{U}(\bP)] +  \bS_V \bM \bS_U  - \D \bS_U [\tilde{V}(\bP)] - \bS_U \bM \bS_V - \D\bS_V[\tilde{U}(\bP)] - \bS_U \bM \bS_V \\
    &\quad+ \D\bS_V[\tilde{U}(\bP)] + \bS_V\bM \bS_U \\
    = \, &2 (\bS_V \bM \bS_U - \bS_U \bM \bS_V),
\end{align*}
and the vertical projection is 
\begin{equation*}
    ([\tilde{U}, \tilde{V}](\bP))^\gV = \bM^{-1/2} \L_{\bM^{-1}, \pi(\bP)^{-1}}[2 \bS_V \bM \bS_U - 2\bS_U \bM \bS_V] \bM^{-1/2} \bP^{-\top}.
\end{equation*}
To study the trace norm of the vertical projection, we denote 
\begin{equation*}
    \bL := \L_{\bM^{-1}, \pi(\bP)^{-1}}[2\bS_V \bM \bS_U - 2\bS_U \bM\bS_V].
\end{equation*}
Then, from the definition of generalized Lyapunov operator, 
\begin{align*}
    &\bP^\top \bM^{-1/2} \bL \bM^{-1/2} \bP^{-\top} + \bP^{-1} \bM^{-1/2} \bL \bM^{-1/2} \bP \\
    &= 2\bP^\top \bM^{1/2} (\bS_V \bM \bS_U - \bS_U \bM\bS_V) \bM^{1/2} \bP \\
    &= 2 \tilde{V}(\bP)^\top \tilde{U}(\bP) - 2 \tilde{U}(\bP)^\top \tilde{V}(\bP).
\end{align*}
Now, consider the singular value decomposition of $\bP = \bU \bSigma \bV^\top$ with the singular values sorted decreasingly. Denote $\bC := \bV^\top (\tilde{V}(\bP)^\top \tilde{U}(\bP) - \tilde{U}(\bP)^\top \tilde{V}(\bP)) \bV$. This yields
\begin{align}
    2\bC  = \bSigma \bU^\top \bM^{-1/2} \bL \bM^{-1/2} \bU \bSigma^{-1}  + \bSigma^{-1} \bU^\top \bM^{-1/2} \bL \bM^{-1/2} \bU \bSigma. \label{curv_eq_1} 
\end{align}
Denote $\tilde{\bL} := \bU^\top \bM^{-1/2} \bL \bM^{-1/2} \bU$. Result \eqref{curv_eq_1} indicates $(\sigma_i \sigma_j^{-1} + \sigma_i^{-1} \sigma_j)\tilde{\bL}_{ij} = 2\bC_{ij}$.   Hence, $\bM^{-1/2}\bL \bM^{-1/2} = \bU \tilde{\bL} \bU^\top$ and 
\begin{align*}
    \|([\tilde{U}, \tilde{V}](\bP))^\gV \|_2^2 = \| \bU \tilde{\bL} \bU^\top \bP^{-\top} \|_2^2 &= \| \bU \tilde{\bL} \bSigma^{-1} \bV^\top \|_2^2 \\
    &= \| \tilde{\bL} \bSigma^{-1} \|_2^2 \\
    &= \sum_{i,j} \frac{4\bC_{ij}^2}{\sigma_j^2 (\sigma_i \sigma_j^{-1} + \sigma_i^{-1} \sigma_j)^2}.
\end{align*}
Based on Theorem \ref{curvature_theorem}, the proof is complete by noticing $\M_{\rm gl}$ has zero curvature and choosing orthonormal tangent vectors $\tilde{U}(\bP), \tilde{V}(\bP)$ without loss of generality.
\end{proof}

\subsection{Proof of Proposition \protect\ref{prop_bounds_sect_curvature}}

We now compute the bounds for the sectional curvature following \cite{massart2019curvature}. We need the following lemma, which bounds the skew operation of matrix product. 

\begin{lemma}[Lemma 2 in \cite{massart2019curvature}]
\label{lemma_bound_skew}
For arbitrary matrices $\bA, \bB \in \sR^{n \times n}$ with $\| \bA\|_2 = \| \bB\|_2 = 1$, we have $\| \bA^\top \bB - \bB^\top \bA \|^2_2 \leq 2$.
\end{lemma}

\begin{proof}[Proof of Proposition \ref{prop_bounds_sect_curvature}]
It is clear when $\bC = \bzero$, the sectional curvature is zero, which happens when for example, $\bS_U = \bM^{-1}, \bS_V = \bS$ for arbitrary symmetric $\bS$. This holds even when $\tilde{U}(\bP), \tilde{V}(\bP)$ are not orthonormal. 

Also, we have
\begin{align*}
    &K(U(\pi(\bP)), V(\pi(\bP))) \\
    &= \sum_{i,j} \frac{3 \bC_{ij}^2}{\sigma^2_j(\sigma_i \sigma_j^{-1} + \sigma_i^{-1} \sigma_j)^2} = \sum_{i,j} \frac{3 \sigma_i^2 \bC_{ij}^2}{(\sigma_i^2 + \sigma_j^2)^2} =  \frac{3\sum_{i > j} (\sigma_i^2 +\sigma_j^2) \bC_{ij}^2}{(\sigma_i^2 + \sigma_j^2)^2} \\
    &= \sum_{i > j} \frac{3\bC_{ij}^2}{\sigma_i^2 + \sigma_j^2} 
    \leq \frac{3}{2(\sigma_n^2+\sigma^2_{n-1})} \| \bC \|_2^2 \leq \frac{3}{\sigma_n^2 + \sigma^2_{n-1}},
\end{align*}
where we notice $\bC$ is skew-symmetric and apply Lemma \ref{lemma_bound_skew}. To verify the choice of $\tilde{U}(\bP)$, $\tilde{V}(\bP)$ that achieves the maximum curvature, we first see
\begin{align*}
    \trace(\tilde{U}(\bP)^\top \tilde{V}(\bP)) &= \frac{\trace(\bSigma^{-2} (\bE_{\{n-1, n-1 \}} - \bE_{\{n,n\}})\bE_{\{n,n-1 \}}  )}{2 (\sigma_n^{-2} + \sigma_{n-1}^{-2}) } \\
    &= \frac{\trace(\bSigma^{-2} (\be_{n-1}\be_n^\top - \be_n \be_{n-1}^\top))}{{\sigma_n^{-2} + \sigma_{n-1}^{-2}}} = 0, \\
    \trace(\tilde{U}(\bP)^\top \tilde{U}(\bP))= \trace(\tilde{V}(\bP)^\top \tilde{V}(\bP)) &= \frac{\trace(\bSigma^{-2} (\be_{n} \be_n^\top + \be_{n-1} \be_{n-1}^\top))}{\sigma_n^{-2} + \sigma_{n-1}^{-2}} =1, 
\end{align*}
which shows $\tilde{U}(\bP), \tilde{V}(\bP)$ are orthonormal. 

Also, we have 
\begin{align*}
    \bC &= \bV^\top (\tilde{V}(\bP)^\top \tilde{U}(\bP) - \tilde{U}(\bP)^\top \tilde{V}(\bP)) \bV \\
    &= \frac{\bE_{\{ n,n-1\}} \bSigma^{-2} (\bE_{\{n-1,n-1 \}} - \bE_{\{ n,n\}}) - (\bE_{\{ n-1,n-1\}} - \bE_{\{ n,n\}})\bSigma^{-2} \bE_{\{ n, n-1\}} }{2(\sigma_n^{-2} + \sigma_{n-1}^{-2})} \\
    &= \be_{n}\be_{n-1}^\top - \be_{n-1} \be_{n}^\top.
\end{align*}
This leads to the maximum sectional curvature as $\sum_{i > j} \frac{3\bC^2_{ij}}{\sigma_i^2 + \sigma_j^2} = \frac{3}{\sigma_n^2 + \sigma_{n-1}^2}$.
\end{proof}

\subsection{Proof of Proposition \protect\ref{operator_inequality_major}}

\begin{proof}[Proof of Proposition \ref{operator_inequality_major}]
From the expression of GBW geodesic, we have
\begin{align*}
    \gamma(t) = \, &(1-t)^2\bX + t^2 \bY + t(1-t) \left(  (\bY \bM^{-1}\bX \bM^{-1})^{1/2}\bM + \bM (\bM^{-1}\bX \bM^{-1} \bY)^{1/2} \right)\\
    = \, &(1-t)^2\bX + t^2 \bY + t(1-t) \Big(  \bY^{1/2} (\bY^{1/2} \bM^{-1}\bX \bM^{-1} \bY^{1/2})^{1/2} \bY^{-1/2} \bM \\
    &+ \bM \bY^{-1/2} (\bY^{1/2} \bM^{-1}\bX \bM^{-1} \bY^{1/2})^{1/2} \bY^{1/2} \Big) \\
    = \, & \bM \bY^{-1/2} \Big( (1-t)^2 (\bY^{1/2} \bM^{-1} \bX \bM^{-1} \bY^{1/2}) + t^2 (\bY^{1/2} \bM^{-1} \bY \bM^{-1} \bY^{1/2})  \\
    &+ t(1-t) \bY^{1/2}\bM^{-1} \bY^{1/2} (\bY^{1/2} \bM^{-1} \bX \bM^{-1} \bY^{1/2})^{1/2} \\
    &+ t(1-t) (\bY^{1/2} \bM^{-1} \bX \bM^{-1} \bY^{1/2})^{1/2} \bY^{1/2} \bM^{-1} \bY^{1/2}
    \Big) \bY^{-1/2} \bM \\
    = \, &\bM \bY^{-1/2} \Big(  (1-t) (\bY^{1/2} \bM^{-1} \bX \bM^{-1} \bY^{1/2})^{1/2} + t (\bY^{1/2} \bM^{-1} \bY^{1/2}) \Big)^2 \bY^{-1/2} \bM \\
    \preceq \, & \bM \bY^{-1/2} \Big( (1-t) (\bY^{1/2} \bM^{-1} \bX \bM^{-1} \bY^{1/2}) + t (\bY^{1/2} \bM^{-1} \bY \bM^{-1} \bY^{1/2}) \Big) \bY^{-1/2} \bM \\
    = \, &(1-t) \bX +t\bY,
\end{align*}
where the second equality follows from the property of geometric mean $(\bA \bB)^{1/2} = \bA(\bA^{-1} \bB)^{1/2} = \bA^{1/2}( \bA^{1/2} \bB \bA^{1/2})^{1/2} \bA^{-1/2}$. 
\end{proof}

\subsection{Proof of Theorem \protect\ref{theorem_barycenter}}

\begin{proof}[Proof of Theorem \ref{theorem_barycenter}]
First we see 
\begin{equation*}
    F(\bA) = \sum_{l=1}^N w_l \trace(\bM^{-1} \bX_l) + \sum_{l=1}^N w_l \trace \Big( \bM^{-1} \bA - 2(\bX_l^{1/2} \bM^{-1} \bA \bM^{-1} \bX_l^{1/2})^{1/2} \Big).
\end{equation*}
Thus to show strict convexity of $F(\bA)$, we only need to show 
\begin{equation*}
    S(\bA) = \trace(\bX_l^{1/2} \bM^{-1} \bA \bM^{-1} \bX_l^{1/2})^{1/2}
\end{equation*}
is strictly concave. This is true because $\trace(\bX)^{1/2}$ is strictly concave. See proof in \cite{bhatia2019bures,bellman1968some}.

By first-order stationarity, we need to find the derivative of $F(\bA)$. First we write $S(\bA) = \trace( (h \circ \phi)(\bA))$, where $h(\bA) = \bA^{1/2}$ and $\phi(\bA) = \bX_l^{1/2} \bM^{-1} \bA \bM^{-1} \bX_l^{1/2}$. Recall that $\D h(\bX)[\bU] = \L_{\bX^{1/2}}[\bU]$ by the derivative of the inverse function law \cite{malago2018wasserstein,bhatia2019bures}. Thus by chain rule, 
\begin{align*}
    \D S(\bA)[\bU] &= \trace \Big( (\D h(\phi(\bA)) \circ \D\phi(\bA))[\bU] \Big)\\
    &= \trace \Big( \L_{ (\bX_l^{1/2} \bM^{-1} \bA \bM^{-1} \bX_l^{1/2})^{1/2} }[\bX_l^{1/2} \bM^{-1} \bU \bM^{-1} \bX_l^{1/2}] \Big) \\
    &= \frac{1}{2}\trace \Big( \bM^{-1} \bX_l^{1/2} (\bX_l^{1/2} \bM^{-1} \bA \bM^{-1} \bX_l^{1/2})^{-1/2} \bX_l^{1/2} \bM^{-1} \bU \Big)\\
    &= \frac{1}{2} \trace \Big( (\bA^{-1} \bM \bX_l^{-1} \bM)^{1/2} \bM^{-1} \bX_l \bM^{-1} \bU \Big) \\
    &= \frac{1}{2} \trace \Big( \bA^{-1} \# (\bM^{-1} \bX_l \bM^{-1}) \bU \Big),
\end{align*}
where the second equality follows from $\trace(\L_\bX[\bU]) = \frac{1}{2} \trace(\bX^{-1} \L_\bX[\bU] \bX + \L_{\bX}[\bU]) = \frac{1}{2}\trace(\bX^{-1} \bU)$ and the third equality is due to \eqref{important_inter_lemma}. 

Hence, $\D F(\bA) [\bU] = \sum_l w_l \trace\Big( \bM^{-1}\bU - \bA^{-1}\# (\bM^{-1} \bX_l\bM^{-1}) \bU \Big)$. From the first-order optimality of the convex function $F(\bA)$, i.e. $\D F(\bA)[\bU] = \bzero$ for all $\bU$, the unique minimizer $A(\bX_{1:N}, \bw)$ satisfies $\bM^{-1} = \sum_{l=1}^N w_l \, \bA^{-1} \# (\bM^{-1} \bX_l \bM^{-1})$, which is equivalent to $\bA^{1/2} \bM^{-1} \bA^{1/2} = \sum_{l=1}^N w_l \, (\bA^{1/2} \bM^{-1} \bX_l \bM^{-1} \bA^{1/2} )^{1/2}$.
\end{proof}

\subsection{Proof of Theorem \protect\ref{fixed_point_iter_GBW}}

\begin{proof}[Proof of Theorem \ref{fixed_point_iter_GBW}]
First by convexity of the matrix square,
\begin{equation*}
    K(\bA) \leq \bM\bA^{-1/2} (\sum_{l=1}^N w_l \bA^{1/2} \bM^{-1} \bX_l \bM^{-1} \bA^{1/2} ) \bA^{-1/2} \bM \leq \sum_{l=1}^N \bX_l. 
\end{equation*}
Hence, $K(\bA)$ is bounded in $\sS_{++}^n$. Also, we claim $F(\bA_{t+1}) \leq F(\bA_t)$, where $F(\bA)$ is the objective function defined in \eqref{gbw_barycenter}. To see this, first we recall from Proposition \ref{prop_ot_plan}, the optimal transport map between two zero-mean Gaussians is $\bT_{\bX \xrightarrow{} \bY} = \bM(\bX^{-1} \# (\bM^{-1} \bY \bM^{-1}))$ with $\bX, \bY$ the respective covariance matrices. Now suppose $\ba \in \sR^n$ is a random Gaussian vector with mean zero and covariance $\bA$ and define $\bx_l = \bT_{\bA \xrightarrow{} \bX_l} \,\ba$. From Proposition \ref{prop_ot_plan}, $\bx_l$ is Gaussian distributed with covariance $\bX_l$ and
\begin{equation*}
    F(\bA) = \sum_{l=1}^N w_l \, d^2_{\rm gbw}(\bA, \bX_l) = \sum_{l=1}^N w_l \, \mathbb{E}\| \ba - \bx_l \|^2_{\bM^{-1}}.
\end{equation*}
In addition, we verify that $\bT_{\bA \xrightarrow{} K(\bA)} = \sum_{l=1}^N w_l \bT_{\bA \xrightarrow{} \bX_l}$. That is, 
\begin{align*}
    \bT_{\bA \xrightarrow{} K(\bA)} &= \bM(\bA^{-1} \# (\bM^{-1} K(\bA) \bM^{-1})) \\
    &= \bM \Big( \bA^{-1} \# \big( \bA^{-1/2} \big( \sum_{l=1}^N w_l \, (\bA^{1/2} \bM^{-1} \bX_l \bM^{-1} \bA^{1/2} )^{1/2} \big)^2 \bA^{-1/2} \big) \Big) \\
    &= \bM \Big( \bA^{-1/2} ( \sum_{l=1}^N w_l (\bA^{1/2} \bM^{-1} \bX_l \bM^{-1} \bA^{1/2})^{1/2} ) \bA^{-1/2} \Big) \\
    &= \sum_{l=1}^N w_l \, \bM \Big( \bA^{-1/2} (\bA^{1/2} \bM^{-1} \bX_l \bM^{-1} \bA^{1/2})^{1/2} \bA^{-1/2} \Big) = \sum_{l=1}^N w_l \bT_{\bA \xrightarrow{} \bX_l}.
\end{align*}
Denote $\bar{\bx} := \sum_{l=1}^N w_l \bx_l$. Then
\begin{equation*}
    d^2_{\rm gbw}(\bA, K(\bA)) = \mathbb{E} \|\ba - \bT_{\bA \xrightarrow{} K(\bA)} \,\ba  \|^2_{\bM^{-1}} = \mathbb{E} \|\ba - \sum_{l=1}^N w_l \bx_l \|^2_{\bM^{-1}} = \mathbb{E}\| \ba - \bar{\bx} \|^2_{\bM^{-1}}.
\end{equation*}
Notice $\bar{\bx} = \bT_{\bA \xrightarrow{} K(\bA)} \,\ba$ is also a zero-mean Gaussian random vector with covariance $K(\bA)$. It follows that $d^2_{\rm gbw} (K(\bA), \bX_l) \leq \mathbb{E}\| \bar{\bx} - \bx_l \|^2_{\bM^{-1}}$. Next recall the variance formula for Euclidean random vector, i.e. ${\rm Var}(\by) = \mathbb{E}\| \by - \mathbb{E}[\by] \|^2 = \mathbb{E} \| \by \|^2 - \| \mathbb{E}[\by] \|^2 = \mathbb{E} \| \bx - \by \|^2 - \| \bx - \mathbb{E}[\by]\|^2$, for arbitrary $\bx$. The analogue under Mahalanobis distance and finite average also holds, i.e., $\sum_{l=1}^N w_l \| \bx_l - \bar{\bx}  \|^2_{\bM^{-1}} = \sum_{l=1}^N w_l \| \ba - \bx_l \|^2_{\bM^{-1}} - \| \ba - \bar{\bx} \|^2_{\bM^{-1}}$. Finally, based on these results, we have 
\begin{align*}
     F(K(\bA)) = \sum_{l=1}^N w_l \, d^2_{\rm gbw}( K(\bA), \bX_l) &\leq \sum_{l=1}^N w_l \, \mathbb{E} \| \bar{\bx} - \bx_l \|^2_{\bM^{-1}} \\
     &= \sum_{l=1}^N w_l \,\mathbb{E} \|\ba - \bx_l \|^2_{\bM^{-1}} - \mathbb{E} \| \ba - \bar{x} \|^2_{\bM^{-1}} \\
     &\leq F(\bA) - d^2_{\rm gbw}(\bA, K(\bA)).
\end{align*}
This suggests $F(K(\bA)) \leq F(\bA)$ and hence together with the boundedness of $K(\bA)$, the sequence $\bA_{t}$ converges. In the limit, we shall observe $F(K(\bA_t)) = F(\bA_t)$ when $t \xrightarrow{} \infty$ and thus $d^2(\bA, K(\bA)) = 0$. From the definition of $K(\bA)$ and the optimality condition, we conclude the limit point is $A(\bX_{1:N}, \bw)$. 
\end{proof}

\section{Additional results and proofs for Section \protect\ref{sec:applications}}

\subsection{Geodesic convexity}
\label{sect_geoedesic_convex}
Geodesic convexity is a generalization of standard convexity in the Euclidean space. It plays a crucial role in Riemannian optimization problems, where for geodesic convex problems, the convergence rates have been shown to be superior in many cases \cite{sra2015conic,zhang2016first}. Consequently, geodesic convexity has been exploited to develop better algorithms for machine learning applications such as Gaussian mixture models \cite{hosseini2020alternative} and metric learning \cite{zadeh2016geometric}. 
% Recently, geodesic convexity has allowed to develop provably accelerated algorithms on manifolds \cite{ahn2020nesterov}. 
Below, we show some interesting classes of objective functions for SPD matrices that are geodesic convex under the GBW geometry.

A \textit{geodesic convex set} $\mathcal{X} \subseteq \M$ requires, for any $x, y \in \mathcal{X}$, the distance minimizing geodesic $\gamma$ connecting the two points lie entirely in the set. A function $f: \mathcal{X} \xrightarrow{} \sR$ is called \textit{geodesic convex} if, for any $x, y \in \mathcal{X}$, it satisfies that, for all $t \in [0,1]$, $f(\gamma(t)) \leq (1-t)f(x) + t f(y)$.

% The function is \textit{strictly geodesic convex} if the equality only holds if $t=0,1$. 

%\alertPJ{Here we should discuss the utility of geodesic convexity. It is not clear why we are discussing it. Also can we prove that if a function is geodesic convex in BW, then it will be geodesic convex in GBW?} 

\begin{proposition}
\label{prop_geodesic_convex}
Suppose $\bA \in \sS_{+}^n$, the set of $n\times n$ semi-definite matrices, and let $\lambda^\downarrow : \sS_{++}^n \xrightarrow{} \sR^n_{+}$ be the eigenvalue map that is decreasingly sorted and $h: \sR_{+} \xrightarrow{} \sR$ be a monotonically increasing and convex function. Then, {{the following}} functions $f_1(\bX) = \trace(\bX\bA)$, $f_2(\bX) = \trace(\bX\bA\bX)$, $f_3(\bX) = -\log\det(\bX)$, $f_4(\bX) = \sum_{j=1}^k h(\lambda_j^\downarrow(\bX))$, $k \in [1,n]$, are geodesic convex under the GBW geometry for any choice of $\bM$.
\end{proposition}

\subsection{Proof of Proposition \protect\ref{riem_grad_hess_gbw}}

Given a function $f:\M \xrightarrow{} \sR$, the Riemannian gradient at $x \in \M$, denoted by ${\rm grad}f(x)$, is the unique tangent vector satisfying $\langle {\rm grad}f(x), u \rangle_{x} = \D_u f(x)$, for any $u \in T_x\M$. $\D_u f(x)$ is the directional derivative. Riemannian Hessian at $x$, $\hess f(x): T_x\M \xrightarrow{} T_x\M$ is defined as the Levi-Civita derivative of the Riemannian gradient, i.e., $\nabla \grad f(x)$.

\begin{proof}[Proof of Proposition \ref{riem_grad_hess_gbw}]
For the Riemannian gradient, we require
\begin{equation*}
    \trace(\nabla f(\bX) \bV) = \frac{1}{2}\trace(\L_{\bX, \bM}[\grad f(\bX)] \bV)
\end{equation*}
for any $\bV \in T_\bX\M_{\rm gbw}$. Thus, we have $\grad f(\bX) = \L_{\bX, \bM}^{-1}[2\nabla f(\bX)]= 2\bX \nabla f(\bX) \bM + 2\bM \nabla f(\bX) \bX$. 

For the Riemannian Hessian, we have for any $\bU \in T_\bX\M$
\begin{align}
    &\hess f(\bX)[\bU] = \nabla_{\bU} \grad f(\bX) \nonumber\\
    = \, &\D_{\bU}\grad f(\bX) - \{ \bM \L_{\bX, \bM}[\grad f(\bX)] \bU \}_{\rm S} - \{ \bM \L_{\bX, \bM}[\bU] \grad f(\bX) \}_{\rm S} \nonumber\\
    &+ \{ \bX \L_{\bX, \bM}[\grad f(\bX)] \bM \L_{\bX, \bM}[\bU] \bM + \bX \L_{\bX, \bM}[\bU] \bM \L_{\bX, \bM}[\grad f(\bX)]\bM \}_{\rm S} \nonumber\\
    = \, &\D_{\bU}\grad f(\bX) + \{ 4\bX \{\nabla f(\bX)\bM\L_{\bX, \bM}[\bU] \}_{\rm S}\bM \}_{\rm S} - \{2 \bM \nabla f(\bX) \bU \}_{\rm S} \nonumber\\
    &- \{ \bM \L_{\bX, \bM}[\bU] \grad f(\bX) \}_{\rm S} \label{temp_5_hess},
\end{align}
where we use $\L_{\bX, \bM}[\bX\bU\bM + \bM \bU \bX] = \bU$. Now we compute $\D_\bU \grad f(\bX)$, which is
\begin{align}
    \D_{\bU}\grad f(\bX) &= 2\D_\bU (\bX \nabla f(\bX) \bM + \bM \nabla f(\bX) \bX) \nonumber\\
    &= 2 \bU \nabla f(\bX) \bM + 2 \bX \nabla^2f(\bX)[\bU] \bM + 2 \bM \nabla^2 f(\bX)[\bU] \bX + 2\bM \nabla f(\bX) \bU \nonumber\\
    &= 4\{ \bM \nabla f(\bX) \bU \}_{\rm S} + 4 \{ \bM \nabla^2f(\bX)[\bU] \bX \}_{\rm S}.\label{temp_4_hess}
\end{align}
Combining \eqref{temp_4_hess} with \eqref{temp_5_hess} completes the proof.
\end{proof}

\subsection{Proof of Proposition \protect\ref{prop_geodesic_convex}}

\begin{proof}[Proof of Proposition \ref{prop_geodesic_convex}]
To prove geodesic convexity for $f_1, f_2$, we require a second-order characterization of geodesic convexity. That is, a twice continuously differentiable function $f$ is geodesic convex if $\frac{d^2f(\gamma(t))}{dt^2} \geq 0$ for all $t \in [0,1]$. Now recall from Proposition \ref{proposition_geodesic} and the simplification in \eqref{geodesic_simplify_bw}, the geodesic for GBW shares the same form as BW except for the value of polar factor $\bU$. Nevertheless, the non-negativity of second-order derivatives does not depend on the choice of $\bU$ according to the proof of Proposition 1 in \cite{han2021riemannian}. Hence, we can follow the exact proof to show $f_1$ and $f_2$ are geodesic convex on the GBW geometry. 

For $f_3$, we have
\begin{align*}
    \log\det(\gamma(t)) &= 2\log\det((1-t) \bX^{1/2} + t\bY^{1/2}\bU ) \\
    &= 2\log\det( ((1-t) \bM + t \bY^{1/2} \bU \bX^{-1/2} \bM )  \bM^{-1} \bX^{1/2} ) \\
    &\geq 2(1-t)\log\det(\bM) + 2t\log\det( \bY^{1/2} \bU \bX^{-1/2} \bM ) +  2\log\det(\bM^{-1}) \\
    &\qquad+ 2\log\det(\bX^{1/2}) \\
    &= 2t\log\det(\bY^{1/2}) -2t\log\det(\bX^{1/2}) + 2 \log\det(\bX^{1/2}) \\
    &= (1-t) \log\det(\bX) + t\log\det(\bY),
\end{align*}
where the inequality is due to the concavity of log-det on SPD matrices and from Lemma \ref{lemma_U}, we see $\bY^{1/2} \bU \bX^{-1/2} \bM \succeq \bzero$.

Finally for $f_4$, the geodesic convexity simply follows from the result of $\bX \star_t \bY \preceq (1-t)\bX + t\bY$ in Proposition \ref{operator_inequality_major} and Theorem 2.3 in \cite{sra2015conic}.
\end{proof}

% \subsection{Proof of Proposition \protect\ref{prop_robust_GBW_distance}}

\section{Additional developments on the GBW geometry}

\subsection{Results on geometric interpolation and barycenter}
The geometric mean between symmetric positive definite matrices $\bX$ and $\bY$ under the GBW geometry is the mid-point $\gamma(1/2)$ on the geodesic $\gamma$ that connects $\bX$ to $\bY$. Following the notation in \cite{bhatia2019bures}, we denote the interpolation of the generalized BW geodesic as $\bX \star_t  \bY := \gamma(t)$ derived in Proposition \ref{proposition_geodesic}. We show an operator inequality between the interpolation on GBW and convex combination on the Euclidean space.

\begin{proposition}(Operator inequality)
\label{operator_inequality_major}
For any $\bX, \bY \in \sS_{++}^n$, we have $\bX  \star_t \bY \preceq (1-t) \bX + t\bY$, for $t \in [0,1]$, where $\preceq$ denotes the L\"owner partial order.
\end{proposition}

An immediate result from this proposition is that $\log\det(\bX \star_t \bY) \leq \log\det((1-t) \bX + t\bY)$. This has implication in the application of Diffusion Tensor Imaging, where the larger determinant of interpolation of SPD matrices indicates the larger diffusion, known as the swelling effect, which is physically undesirable \cite{arsigny2006log,pennec2006riemannian}. Because $\log\det$ is geodesic concave on $\M_{\rm gbw}$ (Proposition \ref{prop_geodesic_convex}), the swelling effect still exists (unlike the affine-invariant or the log-Euclidean geometry), but the level of adverse effect is smaller compared to Euclidean metric.

Given a set of SPD matrices $\{\bX_{l}\}_{l = 1}^N$, the barycenter (or Riemannian center of mass) learning problem is 
\begin{equation}
    \min_{\bA \in \sS_{++}^n } F(\bA) := \sum_{l=1}^N w_l d^2_{\rm gbw}(\bX_l, \bA), \label{gbw_barycenter}
\end{equation}
with $\sum_{l=1}^N w_l = 1$. This is an extension of the Wasserstein barycenter of Gaussian measures \cite{agueh2011barycenters,bhatia2019bures}. Denote the minimizer as $A(\bX_{1:N}, \bw) := \arg\min_{\bA \in \sS_{++}^n} F(\bA)$. We can show, from matrix theory, that the minimizer is unique and is the solution to a specific nonlinear matrix equation. This generalizes the results in \cite{bhatia2019bures}.

\begin{theorem}[Generalization of the result from \cite{bhatia2019bures}]
\label{theorem_barycenter}
The function $F(\bA)$ is strictly (Euclidean) convex in the convex cone of $\sS_{++}^n$, which admits a unique GBW barycenter $A(\bX_{1:N}, \bw)$. The barycenter is the solution to the equation $\bA^{1/2} \bM^{-1} \bA^{1/2} = \sum_{l=1}^N w_l \, (\bA^{1/2} \bM^{-1} \bX_l \bM^{-1} \bA^{1/2} )^{1/2}.$
\end{theorem}

Next, we show how to compute the barycenter by a fixed point iteration similar in \cite{bhatia2019bures,alvarez2016fixed}. Let 
$$K(\bA) := \bM \bA^{-1/2} \big( \sum_{l=1}^N w_l \, (\bA^{1/2} \bM^{-1} \bX_l \bM^{-1} \bA^{1/2} )^{1/2} \big)^2 \bA^{-1/2} \bM,$$ 
and perform the iteration update by $\bA_{t+1} = K(\bA_t)$. We can show this update converges to $A(\bX_{1:N}, \bw)$, formalized in the following Theorem.

\begin{theorem}
\label{fixed_point_iter_GBW}
Initialize $\bA_0 \in \sS_{++}^n$ randomly and consider the update $\bA_{t+1} = K(\bA_t)$. Then $\lim_{t \xrightarrow{} \infty} \bA_{t} = A(\bX_{1:N}, \bw)$.
\end{theorem}

\subsection{Robust GBW distance} \label{sec:rgbw}
% \alertAH{Move to appendix.}
% The relationship between the Bures-Wasserstein distance and the $L_2$-Wasserstein distance has been well-studied in~\cite{bhatia2019bures,van2020bures}. 

In this section, we show that the connection of the GBW distance with a class of projection robust Wasserstein distances between zero-centered Gaussians. This may be of independent interest. 

Robust Wasserstein distances~\cite{paty2019subspace,huang2021projection} may help mitigate the sample complexity of Wasserstein distances, which may grow  exponentially in dimension~\cite{dudley69a,fournier15a,weed19a}. 
% the exponentially growing complexity in dimension to compute the Wasserstein distance. 
Given two $n$-dimensional measures $\mu, \nu$, the projection robust Wasserstein distance~\cite{paty2019subspace,huang2021projection} is computed as follows:
\begin{equation*}
    \gP_{d}(\mu, \nu) = \sup_{\bW: \bW^\top \bW = \bI} \inf_{\gamma \sim \Gamma(\mu, \nu)} \int \| \bW^\top (\bx - \by)\|^2 d\gamma(\bx, \by),
\end{equation*}
where $\bW \in \sR^{n \times d}$ ($d \leq n$) is a projection matrix which is learned over the given samples. When $\mu$ and $\nu$ are zero-centered Gaussians with covariance matrices $\bX$ and $\bY$, respectively, this reduces to 
\begin{align}
    &\gP_d(\mu=\gN(\bzero,\bX), \nu=\gN(\bzero,\bY)) = \label{eqn:zero-centered-gaussians1}\\ 
    &\sup_{\bW:\bW^\top\bW=\bI} \trace(\bW\bW^\top \bX) + \trace(\bW\bW^\top \bY) - 2 \trace(\bX^{1/2} \bW\bW^\top \bY \bW \bW^\top \bX^{1/2})^{1/2}\nonumber
\end{align}
based on Proposition \ref{prop_wasserstein_dist}. If $\bW^{*}$ is an optimal solution of \eqref{eqn:zero-centered-gaussians1}, we also have the following equivalence: $\gP_d(\mu=\gN(\bzero,\bX), \nu=\gN(\bzero,\bY))=d^2_{\rm gbw}(\bX,\bY)$ for $\bM^{-1}=\bW^{*}(\bW^{*})^\top$. Hence, for a specific choice of $\bM^{-1}$, the GBW distance may be interepreted as a projection robust Wasserstein distance between zero-centered Gaussians.

Based on the above discussion, we now define a class of robust Wasserstein distances $d_{\rm rgbw}$ for $\bM^{-1} \succ \bzero$ as 
\begin{align}
    d^2_{\rm rgbw}(\bX, \bY) &= \max_{\bM^{-1} \in \gC} \, d^2_{\rm gbw}(\bX,\bY) = \max_{\bS \in \gC} \, \trace(\bS \bX) + \trace(\bS \bY) - 2\trace(\bX^{1/2} \bS \bY\bS\bX^{1/2})^{1/2} \label{robust_bw_dist}
\end{align}
for a closed convex set $\gC \subseteq \sS_{++}^n$. We emphasize the maximization of $\bS$ over the set $\gC$. 
%It follows from~(\ref{eqn:zero-centered-gaussians1})~\&~(\ref{robust_bw_dist}) that the projection robust Wasserstein distance for zero-centered Gausssians is a particular case of the robust GBW distance. 
Below we show that \eqref{robust_bw_dist} is a distance metric. 
\begin{proposition}
\label{prop_robust_GBW_distance}
The robust GBW distance \eqref{robust_bw_dist} in the set $\gC \subseteq \sS_{++}^n$ is a distance metric. 
\end{proposition}
\begin{proof}[Proof of Proposition \ref{prop_robust_GBW_distance}]
From (\ref{robust_bw_dist}), we see $d^2_{\rm rgbw}(\bX, \bY) \geq 0$ and is clearly symmetric. The triangle inequality also easily follows as shown below. Let 
\begin{equation}\label{eq:robust_argmax}
   \bS^{*} = \argmax_{\bS \in \gC} d_{\rm gbw}^2(\bX, \bY).
\end{equation}
Therefore, from (\ref{eq:robust_argmax}), we have
\begin{align*}
d_{\rm rgbw}(\bX, \bY) & =  d_{gbw}(\bX, \bY) {\text{ for }} \bS^{*} \\
& \leq d_{gbw}(\bX, \bZ) + d_{gbw}(\bZ, \bY) {\text{ for }} \bS^{*} \text{ as GBW is a distance}\\
& \leq (\max_{\bS_1 \in \gC}   d_{gbw}(\bX, \bZ) {\text{ for }} \bS_1) + (\max_{\bS_2 \in \gC}   d_{gbw}(\bZ, \bY) {\text{ for }} \bS_2) \\
    &= d_{\rm rgbw}(\bX, \bZ) + d_{\rm rgbw}(\bZ, \bY),
\end{align*}
where $\bX$, $\bY$, and $\bZ$ are SPD matrices. Finally, the identity of indiscernibles property is satisfied as the robust GBW distance is based on the GBW distance (which itself satisfies the property). This completes the proof.
\end{proof}

\subsection{GBW geometry and metric learning}
\label{sect_metric_learning}

The problem of metric learning amounts to learning a suitable Mahalanobis (symmetric positive, and possibly, semi-definite) matrix from pairs of similarity and dissimilarity information, e.g., in a classification task \cite{zadeh2016geometric,harandi2017joint,guillaumin2009you,harandi2014manifold,huang2017geometry,huang2015log}.

A particular formulation of interest is based on the objective function proposed in \cite{harandi2017joint}. Specifically, given a set of data-target pairs $\{ \bX_i, t_i\}, \bX_i \in \sS_{++}^n$ and $t_i$ categorical, we define the class adjacency matrix $\bA_{ij} = 1$ if sample $i,j$ are from the same class (i.e., $t_i = t_j$) and $\bA_{ij} = -1$ otherwise. To this end, the objective function is given as 
\begin{equation}\label{eq:metric_learning}
    \min_{\bS \succeq \bzero} \sum_{i,j}^N {\rm log}(1 + {\rm exp}(\bA_{ij} ( \trace(\bS \bX_i) + \trace(\bS \bX_j) - 2\trace(\bX_i^{1/2} \bS \bX_j \bS \bX_i^{1/2})^{1/2} )) ).
\end{equation}
It should be emphasized that the objective function in \eqref{eq:metric_learning} is formulated by directly making use of the BW distance in the objective function of \cite{harandi2017joint}. However, from the definition of the GBW distance \eqref{gbw_distance} between $\bX_i$ and $\bX_j$ and by taking $\bM^{-1} = \bS$, we observe that the problem~\eqref{eq:metric_learning} may be equivalently rewritten as 
\begin{equation*}\label{eq:metric_learning_gbw}
    \min_{\bS \succeq \bzero} \sum_{i,j}^N {\rm log}(1 + {\rm exp}(\bA_{ij}d^2_{\rm gbw}(\bX_i, \bX_j)) ).
\end{equation*}
%the term $ \trace(\bS \bX_i) + \trace(\bS \bX_j) - 2\trace(\bX_i^{1/2} \bS \bX_j \bS \bX_i^{1/2})^{1/2} $ is the same as the GBW distance between $\bX_i$ and $\bX_j$ with $\bM^{-1} = \bS$. 
This suggests that the GBW geometry naturally captures the metric learning properties of the space. Note that $\bS$ can be arbitrary semi-definite matrix and one usually parameterizes $\bS = \bW \bW^\top$, where $\bW$ is a matrix of size $n\times d$. Similar to Section \ref{geometry_aware_pca}, we usually consider $d \ll n$ for practical considerations.

% $\bW^\top \bW = \bI$. This corresponds to a supervised geometry-aware dimensionality reduction method on SPD matrices \cite{harandi2014manifold,huang2017geometry}.  
% \GaoC{Cite some papers from Harandi?}

\section{Additional experiments on geometry-aware principal component analysis (PCA)}
\label{geometry_aware_pca}\label{sect_pca}
% In many machine learning tasks, the covariance matrices generated, say from data points, lie on a high-dimensional SPD manifold, which may not computatianlly efficient in learning tasks. Therefore, a number of works \cite{horev2016geometry,harandi2014manifold,huang2015log} consider reducing the dimensionality of the SPD matrices while maintaining some local or global properties. \alertBM{Needs rephrasing. Connetion with GBW is not apparent. It should be emphasized.}

In this section, we explore the connection of the GBW distance and geometry-aware principal component analysis.

\paragraph{\textbf{Problem formulation:}} Geometry-aware principal component analysis (PCA) for SPD matrices extends the classical PCA to manifolds by maximizing the deviation from the reduced SPD matrices to the reduced barycenter \cite{horev2016geometry,harandi2014manifold,huang2015log}. Using the BW distance, the PCA objective is formulated naturally as the GBW distance between matrices, where $\bM ^{-1}$ is parameterized as $\bW \bW^\top$ with $\bW \in \sR^{n \times d}$. Note that $\bM^{-1}$ is low rank, therefore, does not strictly fall under the generalized metric. Nevertheless, we can make use of the GBW distance expression and substitute low-rank paramterized $\bM^{-1}$.

Consequently, the objective function is
\begin{align*}
    &\max\limits_{\bM^{-1}=\bW \bW^\top: \bW^\top \bW = \bI} \sum_{i = 1}^N d^2_{\rm gbw}(\bX_i ,  \bar{\bX}) = \max_{\bW: \bW^\top \bW = \bI} \sum_{i = 1}^N d^2_{\rm bw}(\bW^\top \bX_i \bW, \bW^\top \bar{\bX} \bW) 
%   & = \max\limits_{\bW : \bW^\top \bW = \bI}  \sum_{i=1}^N ( \trace(\bW^\top \bX_i\bW) + \trace(\bW^\top \bar{\bX}\bW) -2\trace(\bW^\top \bX_i \bW\bW^\top \bar{\bX}\bW)^{1/2} )
\end{align*}
for samples $ \bX_i \in \sS_{++}^n, i =1,\ldots,N$, where $\bar{\bX} = \argmin \sum_{i=1}^N d^2_{\rm bw}(\bX_i, \bC)$ is the barycenter in the original space. 
% $\bar{\bX}$ can be computed via the fixed point iteration algorithm (Theorem \ref{fixed_point_iter_GBW}) by setting $\bM = \bI$. 
The constraint of column orthonormality on $\bW$, i.e., $\bW^\top \bW = \bI$, ensures that $\bW$ projects the covariance matrices onto a $d$-dimensional space. In many practical scenarios, $d$ is often chosen to be much less than $n$, i.e., $d\ll n$.

\paragraph{\textbf{Tasks:}} For the application of geometry-aware PCA, we consider two vision tasks, i.e., image set classification and video-based face recognition. Following the pre-processing steps in \cite{harandi2014manifold,huang2015log}, we treat each vectorized image (or video frame) as a sample in the set and compute the sample covariance to represent the entire image set (or a video). The task is to classify each image set or video represented by a covariance SPD matrix. 

\paragraph{\textbf{Datasets:}} Three real-world datasets are considered, including the MNIST handwritten digits (\texttt{MNIST}) \cite{lecun1998gradient}, ETH-80 object (\texttt{ETH}) \cite{leibe2003analyzing}, and YouTube Celebrities (\texttt{YTC}) \cite{kim2008face} datasets. To process \texttt{MNIST} dataset, we use 42\,000 training samples, and, for each class, we partition the samples into subgroups randomly, each containing 50 images. Then for each subgroup, the covariance matrix is computed. \texttt{ETH} dataset contains image sets of 8 objects, each with 10 subclasses. The 80 subgroups are processed accordingly. \texttt{YTC} is a collection of low-resolution videos of celebrities. Due to the sparsity of the dataset, we only consider 9 persons with video number greater than 15. All images or video frames are resized to $10\times 10$ and the SPD matrix generated as the covariance is of size $100 \times 100$. The statistics of all the considered datasets are in Table~\ref{dataset_statistic}.

\begin{table}[t]
\begin{center}
{\small 
\caption{Summary statistics for \texttt{MNIST}, \texttt{ETH}, \texttt{YTC} datasets}
\label{dataset_statistic}
\setlength{\tabcolsep}{4pt} 
\renewcommand{\arraystretch}{1.3}
\begin{tabular}{c | c c c c c c c c}
\toprule
     & SPD samples & SPD Dim & \# Class      \\
\hline
\texttt{MNIST} & 835 & 100 & 10\\
\texttt{ETH} & 80 & 100 &8 \\
\texttt{YTC} & 194 & 100 &9\\
\bottomrule
\end{tabular}
}
\end{center}
\end{table}

\begin{table}[t]
\begin{center}
{\small 
\caption{Geometry-aware PCA average classification accuracy ($\%$). GBW allows lower dimensional projection with accuracy comparable to that in the original dimension.}
\label{pca_result}
\setlength{\tabcolsep}{4pt} 
\renewcommand{\arraystretch}{1.3}
\begin{tabular}{c | c c c | c c c c c c c}
\toprule
 & \multirow{2}{*}{AI} & \multirow{2}{*}{LE} & \multirow{2}{*}{BW} & \multicolumn{6}{c}{GBW} \\
 &  &  &  & $d=5$ & $d=10$ & $d=30$ & $d=50$ & $d=70$ &$d=90$\\
\hline
\texttt{MNIST} & 100 & 100 & 100 & 99.33 & 100 &100 &100 &100 &100 \\
\texttt{ETH} &76.25  &84.50  &87.75	&80.75	&84.75	&86.75	&88.00	&87.75	&87.75 \\
\texttt{YTC} & 74.70  & 79.00 &76.40 &60.60 &72.40 &76.50 &76.00 & 76.30 & 76.40\\
\bottomrule
\end{tabular}
}
\end{center}
\end{table}

\paragraph{\textbf{Experimental setup:}} As discussed in the above problem formulation, our aim is to find the transformation matrix $\bW \in \sR^{n \times d}$. To validate the effectiveness of dimensionality reduction under the GBW geometry, we perform nearest neighbour classification on the reduced data matrix $\bW^\top \bX_i \bW, i = 1,\ldots,N$. The reduced dimension $d$ is a hyperparameter, and we, therefore, present classification accuracy with $d = \{5, 10, 30, 50, 70, 90\}$. Given that the sample size may be small for some classes, for each class, we take 50\% as the training set and the rest as the test set. Such a random splitting is repeated ten times and we report the average accuracy in Table~\ref{pca_result}, where we also report results with the affine-invariant (AI) and Log-Euclidean (LE)~\cite{arsigny2006log} distances as benchmarks. We use the Riemannian trust region method to solve the maximization problem in Section~\ref{geometry_aware_pca}.

\paragraph{\textbf{Results:}} In Table~\ref{pca_result}, we observe that the classification performance under various choices of $d$ for the GBW distance does not largely degrade, which suggests the global properties of SPD samples can be well-preserved even with a lower-dimensional representation. This also suggests that GBW is a better modeling approach than BW for the geometric PCA problem.

\end{document}